\documentclass{amsart}
\usepackage[utf8]{inputenc}

\usepackage{amssymb}
\usepackage{mathtools}
\mathtoolsset{showonlyrefs=true}
\usepackage{amsthm}
\usepackage{graphicx}
\usepackage{bm}
\usepackage{mathrsfs}

\numberwithin{equation}{section}

\theoremstyle{definition}
\newtheorem{alg}{Algorithm}[section]

\newtheorem{dfn}[alg]{Definition}
\theoremstyle{plain}
\newtheorem{thm}[alg]{Theorem}
\theoremstyle{remark}
\newtheorem{rem}[alg]{Remark}

\allowdisplaybreaks

\begin{document}
\title[Numerical conformal mapping based on DSM]{Bidirectional numerical conformal mapping based on the dipole simulation method}
\author{Koya Sakakibara}
\address[K.~Sakakibara]{Graduate School of Science, Kyoto University, Kitashirakawa Oiwake-cho, Sakyo-ku, Kyoto 606-8502, Japan; RIKEN iTHEMS, 2-1 Hirosawa, Wako-shi, Saitama 351-0198, Japan}
\email{ksakaki@math.kyoto-u.ac.jp}
\begin{abstract}
	Many authors have studied the numerical computation of conformal mappings (numerical conformal mapping), and there are nowadays several efficient numerical schemes.
	Among them, Amano's method offers a straightforward numerical procedure for computing conformal mappings based on the method of fundamental solutions, and it has been applied to several regions with great success.
	However, there are some difficulties in constructing a suitable conjugate harmonic function; therefore, it is required a little of craftmanship.
	In this paper, we construct another numerical scheme for computing conformal mappings based on the dipole simulation method and give mathematical theorems on the approximation error and the arrangements of the singular points.
	Several numerical results are also presented in order to ensure the effectiveness of our proposed method.
\end{abstract}
\maketitle

\section{Introduction}
\label{sec:introduction}

As is well known, a conformal mapping is a fundamental tool in complex analysis.
In some individual cases, we know the explicit formulae of conformal mappings.
However, in most cases, it is hard to obtain explicit formulae; therefore, numerical computation plays essential roles.
In \cite{MR0207240}, Symm developed an integral equation method for computing conformal mappings, in which we have to solve singular Fredholm integral equations of the first kind.
Following this pioneering work, several researchers have developed numerical methods to solve them effectively \cite{MR0301176,MR0396926,MR545882,MR661073,MR615896,MR712114,MR829032,MR829034}.
Aside from these results, Amano has developed novel methods, which we call Amano's method, for computing conformal mappings based on the method of fundamental solutions, in which we do not need to solve any integral equations.
Note that recently Nasser has developed another efficient numerical method which is based on a uniquely solvable boundary integral equation with the generalized Neumann kernel \cite{MR2491542,MR2805493,MR3071407}.
We below explain Amano's method in detail because our proposed numerical scheme is based on it.

\subsection{Amano's method}

We here briefly review Amano's method for bidirectional conformal mapping since it will be the basis of our method.
The term ``bidirectional'' means that we compute a conformal mapping $f$ in a forward direction, from the problem region onto the corresponding canonical region, together with the one $f_*$ in a backward direction, from the canonical region onto the original problem region.
Note that $f_*$ is nothing but the inverse of $f$.
For the sake of simplicity, we consider the case where the problem region $\Omega$ is a Jordan region.
The corresponding canonical region is the unit disk $D_1$ in this case, where $D_\rho$ represents the open disk with radius $\rho$ having the origin as its center: $D_\rho:=\{z\in\mathbb{C}\mid|z|<\rho\}$.
The following explanation is based on the paper \cite{MR1170484}.

Concerning a forward conformal mapping $f$, let us consider the case that $f$ satisfies the following normalizing conditions:
\begin{equation}
	f(z_0)=0,
	\quad
	f'(z_0)>0,
\end{equation}
where $z_0$ is a base point arbitrarily chosen from $\Omega$.
Then we express $f$ in the following form:
\begin{equation}
	f(z)=(z-z_0)\exp[g(z)+\mathrm{i}h(z)].
\end{equation}
$g$ is a harmonic function in $\Omega$ and $h$ is a conjugate harmonic function of $g$.
Since $\partial\Omega$ is a Jordan curve, $f$ can be extended to a homeomorphism from $\overline{\Omega}$ onto $\overline{D}_1$.
Especially, $f(\partial\Omega)=\gamma_1$ holds, where $\gamma_\rho:=\partial D_\rho$ denotes the circle with radius $\rho$ having the origin as its center.
Thus, a harmonic function $g$ can be characterized as a solution to the following potential problem:
\begin{equation}
	\begin{dcases*}
		\triangle g=0&in $\Omega$,\\
		g(z)=-\log|z-z_0|&on $\partial\Omega$.
	\end{dcases*}
	\label{eq:potential_problem_g}
\end{equation}
Therefore, if we can solve the above potential problem for $g$ and can construct its conjugate harmonic function $h$, we can obtain $f$.
In order to do that, the method of fundamental solutions (MFS for short) is applied in Amano's method.

The MFS is a meshfree numerical solver for linear and homogeneous partial differential equations, and its concept is straightforward.
We first choose the singular points $\{\zeta_k\}_{k=1}^{N}\subset\mathbb{C}\setminus\overline{\Omega}$ and the collocation points $\{z_j\}_{j=1}^{N}\subset\partial\Omega$ ``suitably''.
Then we seek an approximate solution $g^{(N)}$ for the problem \eqref{eq:potential_problem_g} of the form
\begin{equation}
	g^{(N)}(z)=\sum_{k=1}^{N}Q_k\log|z-\zeta_k|,
\end{equation}
and coefficients $\{Q_k\}_{k=1}^{N}$ are determined so that they satisfy the following collocation equations:
\begin{equation}
	g^{(N)}(z_j)=-\log|z_j-z_0|
	\quad
	(j=1,2,\ldots,N).
	\label{eq:collocation_equations_g}
\end{equation}
Solving this collocation equations, we obtain $g^{(N)}$, and in principle its conjugate harmonic function $h^{(N)}$ can be constructed as follows:
\begin{equation}
	h^{(N)}(z)=\tilde{h}^{(N)}(z)-\tilde{h}^{(N)}(z_0),
	\quad
	\tilde{h}^{(N)}(z)=\sum_{k=1}^{N}Q_k\arg(z-\zeta_k).
\end{equation}
Note that we have to choose the suitable branch of the argument function, and its practical way has been studied by Amano's group (see for instance \cite{MR1310252,MR1614319,MR2931407}).
We do not explain their method in detail but emphasize that there exists a possibility to develop more straightforward methods than Amano's method.
Anyway, using $g^{(N)}$ and $h^{(N)}$, an approximation $f^{(N)}$ of $f$ can be obtained as
\begin{equation}
	f^{(N)}(z)
	=
	(z-z_0)\exp[g^{(N)}(z)+\mathrm{i}h^{(N)}(z)].
	\label{eq:forward_numerical_conformal_mapping}
\end{equation}

In order to compute a backward conformal mapping $f_*$, we adopt the same strategy for computing a forward conformal mapping $f$.
Namely, we expres $f_*$ as 
\begin{equation}
	f_*(w)=z_0+w\exp[g_*(w)+\mathrm{i}h_*(w)],
\end{equation}
where $g_*$ is a harmonic function in the unit disk $D_1$, and $h_*$ is a conjugate harmonic function of $g_*$.
Using a forward conformal mapping $w=f(z)$, $g_*$ can be characterized as the solution to the following potential problem:
\begin{equation}
	\begin{dcases*}
		\triangle g_*=0&in $D_1$,\\
		g_*(w)=\log|z-z_0|-\log|w|&for $w=f(z)$, $z\in\gamma_1$.
	\end{dcases*}
	\label{eq:backward_potential_problem_g}
\end{equation}
Note that the term $\log|w|$ in the boundary condition theoretically vanishes since $w=f(z)$ is a point on the unit circle.
However, we do not eliminate it since we will consider the case where $\Omega$ is a multiply-connected region later.
Solving the potential problem \eqref{eq:backward_potential_problem_g}, we obtain $g_*$, and subsequently $h_*$ and $f_*$.
In Amano's method, an approximation $g_*^{(N)}$ of $g_*$ is also constructed by the MFS.
Since we need information on the forward conformal mapping in the boundary condition, the corresponding collocation equations become as follows:
\begin{equation}
	g_*^{(N)}(f^{(N)}(z_j))=\log|z_j-z_0|-\log|f^{(N)}(z_j)|
	\quad
	(j=0,1,\ldots,N-1).
\end{equation}
We here again emphasize that the backward conformal mapping $f_*$ is the inverse of the forward conformal mapping $f$.
Therefore, if we can apply some simple approximation technique for holomorphic functions, then we would construct more straightforward methods for computing backward conformal mappings.

\subsection{Contribution of this paper}

We here summarize the concept of this paper.

\subsubsection*{New numerical scheme}

The critical point of Amano's method is to approximate harmonic functions by the MFS, which yields problem on discontinuities of conjugate harmonic function.
An approximate solution by the MFS can be regarded as a discretization of the single-layer potential representation of the solution to the potential problem.
It is well known in the field of the potential theory that single-layer potential representation works well for Neumann boundary value problems, not for Dirichlet boundary value problems, and that double-layer potential representation is suitable for Dirichlet boundary value problems (see for instance \cite{MR1357411}).
Therefore, it is natural to consider that an approximation of double-layer potential representation offers more suitable approximations of harmonic functions, and it is none other than the dipole simulation method (DSM for short), which was first proposed in \cite{MR991024} and some mathematical analysis has been done so far \cite{MR991024,MR3576615}.
Our approach does not yield any problem concerning discontinuities in computing conjugate harmonic function.
See for details $\S$\ref{subsec:DSM} and $\S$\ref{subsec:ForwardDSM}.
Note that the usage of the DSM in computing forward conformal mapping is not new, and it has been used in \cite{ogata2015dipole} to solve potential problems.
However, they did not consider any background on the potential theory and just gave a few results of numerical experiments.
In this paper, we give a theoretical error estimate and clarify some mathematical structure on the arrangements of the singular points (see $\S$\ref{subsec:ErrorEstimate} and $\S$\ref{subsec:Arrangements} for details).

In computing backward conformal mapping, we do not solve the potential problem but directly approximate holomorphic function $f_*$ by the complex dipole simulation method (CDSM for short), which was proposed in \cite{MR3448858}.
The CDSM offers an approximation of a holomorphic function by rational function, and it can be regarded as a discretization of Cauchy integral representation of a holomorphic function.
We directly approximate backward conformal mappings by the CDSM, which offers us a more straightforward expression of backward numerical conformal mapping.

\subsubsection*{Error estimate}

As we mentioned in the above, mathematical analysis of the DSM, which is applied to solve the potential problem \eqref{eq:potential_problem_g}, has been done in \cite{MR991024,MR3576615}, and they state that approximate solution by the DSM uniquely exists and approximation error decays exponentially for the number of points.
In order to estimate the accuracy of our method, it is required to give an error estimate for conjugate harmonic function, which has not been studied in previous studies in numerical conformal mapping by Amano's method.
We give a theoretical estimate of our method for computing conjugate harmonic function by using an argument of Hilbert transform, which states that approximation order of conjugate harmonic function (imaginary part) $h$ is governed by that of harmonic function (real part) $g$.

\subsection{Organization of this paper}

The organizations of this paper are as follows.
In $\S$\ref{sec:SimplyConnected}, we derive the DSM from a double-layer potential representation of a harmonic function and construct a numerical scheme for computing forward numerical conformal mapping.
We then explain the notion of the CDSM and develop a numerical scheme for computing backward numerical conformal mapping.
After that, we prove an error estimate for the conjugate harmonic function.
In $\S$\ref{sec:MultiplyConnected}, we extend our result to the case of the multiply-connected region and explain the effectiveness of our numerical scheme by comparing with Amano's method.
In $\S$\ref{sec:numerical_experiments}, we first establish some mathematical structure on the arrangements of the singular points by a straightforward argument.
We then show several numerical results which exemplify the effectiveness of our method.
In $\S$\ref{sec:ConcludingRemarks}, we summarize this paper and give some concluding remarks.

\section{Numerical conformal mapping for simply-connected region}
\label{sec:SimplyConnected}

\subsection{The dipole simulation method}
\label{subsec:DSM}

We first introduce the DSM from the viewpoint of the potential theory.
Let $\Omega$ be a bounded region in $\mathbb{R}^2$ with smooth boundary $\partial\Omega$.
We then consider the following potential problem:
\begin{equation}
	\begin{dcases*}
		\triangle u=0&in $\Omega$,\\
		u=f&on $\partial\Omega$,
	\end{dcases*}
	\label{eq:potential_problem}
\end{equation}
where $f$ is a given data on $\partial\Omega$.
The MFS offers an approximate solution $u^{(N)}$ for the above problem by a linear combination of logarithmic potentials with singularities being located outside $\Omega$:
\begin{equation}
	u^{(N)}(z)
	=
	\sum_{k=1}^{N}Q_k\log|z-\zeta_k|.
\end{equation}
This approximate solution can be regarded as a discretization of the following single-layer potential representation of the exact solution $u$:
\begin{equation}
	u(z)=\int_{\Gamma}Q(\zeta)\log|z-\zeta|\,\mathrm{d}s(\zeta),
\end{equation}
where $\Gamma$ is some smooth Jordan curve enclosing $\overline{\Omega}$, $Q$ is some appropriate function defined on $\Gamma$, and $\mathrm{d}s$ denotes the line element on $\Gamma$.
Indeed, if we substitute a function
\begin{equation}
	Q^{(N)}(\zeta)=\sum_{k=1}^{N}Q_k\delta(\zeta-\zeta_k)
	\label{eq:QN}
\end{equation}
into $Q$, then we formally obtain
\begin{align}
	\int_{\Gamma}Q^{(N)}(\zeta)\log|z-\zeta|\,\mathrm{d}s(\zeta)
	&=
	\sum_{k=1}^{N}Q_k\int_{\Gamma}\delta(\zeta-\zeta_k)\log|z-\zeta|\,\mathrm{d}s(\zeta)\\
	&=
	\sum_{k=1}^{N}Q_k\log|z-\zeta_k|,
\end{align}
which is nothing but an approximate solution by the MFS.
However, it is usually in the potential theory that the solution $u$ is represented by the double-layer potential (see for instance \cite{MR1357411}):
\begin{equation}
	u(z)=\int_{\Gamma}Q(\zeta)\partial_{n(\zeta)}\log|z-\zeta|\,\mathrm{d}s(\zeta),
\end{equation}
where $\partial_{n(\zeta)}$ denotes the derivative in the (unit) outward normal direction $n(\zeta)\in\mathbb{C}$ at $\zeta\in\Gamma$.
Then it is natural to expect that we can obtain another ``good'' approximate solution by substituting $Q^{(N)}$ defined in \eqref{eq:QN} into $Q$.
We obtain that
\begin{align}
	\int_{\Gamma}Q^{(N)}(\zeta)\partial_{n(\zeta)}\log|z-\zeta|\,\mathrm{d}s(\zeta)
	&=
	\sum_{k=1}^{N}Q_k\int_{\Gamma}\delta(\zeta-\zeta_k)\partial_{n(\zeta)}\log|z-\zeta|\,\mathrm{d}s(\zeta)\\
	&=
	\sum_{k=1}^{N}Q_k\partial_{n(\zeta_k)}\log|z-\zeta_k|.
\end{align}
Since the normal derivatives in the above expression can be computed as
\begin{equation}
	\partial_{n(\zeta_k)}\log|z-\zeta_k|
	=
	\Re\left(\frac{n(\zeta_k)}{z-\zeta_k}\right),
\end{equation}
we finally obtain the following approximate solution:
\begin{equation}
	u^{(N)}(z)
	=
	\sum_{k=1}^{N}Q_k\Re\left(\frac{n_k}{z-\zeta_k}\right),
	\label{eq:DSM_approximate_solution}
\end{equation}
where $n_k:=n(\zeta_k)$ for $k=1,2,\ldots,N$, which we call the dipole moments.
This is nothing but an approximate solution by the DSM.
More precisely, an algorithm of the DSM can be described as follows:

\begin{alg}
	\begin{enumerate}
		\renewcommand{\labelenumi}{(\Roman{enumi})}
		\item Choose the singular points $\{\zeta_k\}_{k=1}^{N}$, the collocation points $\{z_j\}_{j=1}^{N}$ and the dipole moments $\{n_k\}_{k=1}^{N}$ ``suitably''.
		\item Seek an approximate solution $u^{(N)}$ for the potential problem \eqref{eq:potential_problem} of the form \eqref{eq:DSM_approximate_solution} to satisfy the following collocation equations:
			\begin{equation}
				u^{(N)}(z_j)=f(z_j)
				\quad
				(j=1,2,\ldots,N),
			\end{equation}
			which is equivalent to the following linear system:
			\begin{equation}
				G\bm{Q}=\bm{f},
			\end{equation}
			where
			\begin{align}
				&G=\left(\Re\left(\frac{n_k}{z_j-n_k}\right)\right)_{j,k=1,2,\ldots,N}\in\mathbb{R}^{N\times N},\\
				&\bm{Q}=(Q_1,Q_2,\ldots,Q_{N})^{\mathrm{T}}\in\mathbb{R}^N,\\
				&\bm{f}=(f(z_1),f(z_2),\ldots,f(z_{N}))^{\mathrm{T}}\in\mathbb{R}^N.
			\end{align}
	\end{enumerate}
\end{alg}

\subsection{Forward numerical conformal mapping based on the DSM}
\label{subsec:ForwardDSM}

The most important feature of the DSM is that an approximate solution by the DSM can be expressed as a real part of a holomorphic function as follows:
\begin{equation}
	u^{(N)}(z)=\Re\left(\sum_{k=1}^{N}\frac{Q_kn_k}{z-\zeta_k}\right).
	\label{eq:dipole_simulation_method_real_part}
\end{equation}
Therefore one of its conjugate harmonic functions $\tilde{v}^{(N)}$ can be obtained by simply replacing the symbol ``$\Re$" with ``$\Im$'':
\begin{equation}
	\tilde{v}^{(N)}(z)=\Im\left(\sum_{k=1}^{N}\frac{Q_kn_k}{z-\zeta_k}\right).
	\label{eq:conjugate_harmonic_function_DSM}
\end{equation}
This observation leads us to the following algorithm for a numerical conformal mapping from $\Omega$ onto the unit disk $D_1$.

\begin{alg}
	\label{alg:forward_numerical_conformal_mapping}
	\begin{enumerate}
		\renewcommand{\labelenumi}{(\Roman{enumi})}
		\item Choose the singular points $\{\zeta_k\}_{k=1}^{N}$, the collocation points $\{z_j\}_{j=1}^{N}$ and the dipole moments $\{n_k\}_{k=1}^{N}$ ``suitably''.
		\item Seek an approximate solution $g^{(N)}$ for the potential problem \eqref{eq:potential_problem_g} of the form \eqref{eq:DSM_approximate_solution} that satisfies the collocation equations \eqref{eq:collocation_equations_g}.
		\item Construct the conjugate harmonic function $h^{(N)}$ of $g^{(N)}$ as $h^{(N)}=\tilde{h}^{(N)}-\tilde{h}^{(N)}(z_0)$, where $\tilde{h}^{(N)}$ is defined by \eqref{eq:conjugate_harmonic_function_DSM}.
		\item Construct $f^{(N)}$ by \eqref{eq:forward_numerical_conformal_mapping}.
	\end{enumerate}
\end{alg}

We, here again, emphasize that there does not exist any problem on the choice of suitable branch of the argument function.
In this point, we can say that the above algorithm is more straightforward than Amano's method.
In $\S$\ref{sec:numerical_experiments}, we show that the accuracy of our method is almost the same as that of Amano's method.

\subsection{The complex dipole simulation method}

In $\S$\ref{sec:introduction}, we have pointed out that computing the backward conformal mapping $f_*$ is nothing but computing the inverse of the forward conformal mapping $f$.
In order to approximate $f_*$ directly by using boundary correspondence, we apply the CDSM, which is an approximation technique for holomorphic functions by rational functions, which was developed in \cite{MR3448858}.
Its idea is based on the expression \eqref{eq:dipole_simulation_method_real_part} for an approximate solution by the DSM, from which we can expect that the rational function
\begin{equation}
	f^{(N)}(w)=\sum_{k=1}^{N}\frac{Q_k}{w-\xi_k}
	\label{eq:CDSM}
\end{equation}
offers a good approximation for a holomorphic function $f$.
More precisely, an algorithm of the CDSM can be given as follows:

\begin{alg}
	\begin{enumerate}
		\renewcommand{\labelenumi}{(\Roman{enumi})}
		\item Choose the singular points $\{\xi_k\}_{k=1}^{N}$, and the collocation points $\{w_j\}_{j=1}^{N}$ ``suitably''.
		\item Seek an approximation $f^{(N)}$ of $f$ of the form \eqref{eq:CDSM} to satisfy the following collocation equations:
			\begin{equation}
				f^{(N)}(w_j)=f(w_j)
				\quad
				(j=1,2,\ldots,N),
			\end{equation}
			which is equivalent to the following linear system:
			\begin{equation}
				G\bm{Q}=\bm{f},
			\end{equation}
			where
			\begin{align}
				&G=\left(\frac{1}{w_j-\xi_k}\right)_{j,k=1,2,\ldots,N}\in\mathbb{C}^{N\times N},\\
				&\bm{Q}=(Q_1,Q_2,\ldots,Q_{N})^{\mathrm{T}}\in\mathbb{C}^N,\\
				&\bm{f}=(f(z_1),f(z_2),\ldots,f(z_{N}))^{\mathrm{T}}\in\mathbb{C}^N.
			\end{align}
	\end{enumerate}
\end{alg}

\subsection{Backward numerical conformal mapping based on the CDSM}

Using the CDSM, we can directly obtain an approximation $f_*^{(N)}$ of the backward conformal mapping $f_*$ as in the following algorithm.

\begin{alg}
	\label{alg:backward_numerical_conformal_mapping}
	\begin{enumerate}
		\renewcommand{\labelenumi}{(\Roman{enumi})}
		\item Choose the singular points $\{\xi_k\}_{k=1}^{N}$ ``suitably''.
		\item Seek an approximation $f_*^{(N)}$ of $f_*$ of the form \eqref{eq:CDSM} that satisfies the collocation equations:
			\begin{equation}
				f_*^{(N)}(w_j)=z_j,
				\quad
				w_j=f^{(N)}(z_j)
				\quad(j=1,2,\ldots,N),
			\end{equation}
			where $\{z_j\}_{j=1}^{N}$ are the collocation points which are used in Algorithm \ref{alg:forward_numerical_conformal_mapping} to construct $f^{(N)}$.
	\end{enumerate}
\end{alg}

Combining Algorithms \ref{alg:forward_numerical_conformal_mapping} and \ref{alg:backward_numerical_conformal_mapping}, we obtain bidirectional numerical conformal mapping for simply-connected region.

\subsection{Error estimate}
\label{subsec:ErrorEstimate}

Concerning the DSM, it has been proved that if we arrange the singular points, the collocation points and the dipole moments using a peripheral conformal mapping, then an approximation error decays exponentially for $N$ \cite{MR991024,MR3576615}.
As to the CDSM, a similar estimate has been proved in \cite{MR3448858}.
Applying these results, we can immediately show that approximation errors for $g^{(N)}$ and $f_*^{(N)}$ decays exponentially for $N$.
However, these do not tell us the one for $h^{(N)}$, the conjugate harmonic function of $g^{(N)}$.
In this section, we show a simple estimate of an approximation error for $h^{(N)}$.
To this end, we accurately state mathematical results proved in \cite{MR3576615}.

Let $\mathcal{T}$ be a set of finite Fourier series on $S^1=\mathbb{R}/\mathbb{Z}$, that is, $g\in\mathcal{T}$ is represented as
\begin{equation}
	g(\tau)=\sum_{n\in\mathbb{Z}}\hat{g}(n)\mathrm{e}^{2\pi\mathrm{i}n\tau},
	\quad
	\tau\in S^1,
\end{equation}
where $\hat{g}(n)$ are complex numbers of which all but a finite number are zeros.
For $(\epsilon,s)\in(0,+\infty)\times\mathbb{R}$, we define $(\epsilon,s)$-inner product $(\cdot,\cdot)_{\epsilon,s}\colon\mathcal{T}\times\mathcal{T}\rightarrow\mathbb{C}$ and its induced $(\epsilon,s)$-norm $\|\cdot\|_{\epsilon,s}\colon\mathcal{T}\rightarrow\mathbb{R}$ as follows:
\begin{alignat}{2}
	(g,h)_{\epsilon,s}
	&=
	\sum_{n\in\mathbb{Z}}\hat{g}(n)\overline{\hat{h}(n)}\epsilon^{2|n|}\underline{n}^{2s},
	&\quad
	&g,h\in\mathcal{T},\\
	\|g\|_{\epsilon,s}
	&=
	\sqrt{(g,g)_{\epsilon,s}}
	=
	\sqrt{\sum_{n\in\mathbb{Z}}|\hat{g}(n)|^2\epsilon^{2|n|}\underline{n}^{2s}},
	&\quad
	&g\in\mathcal{T},
\end{alignat}
where $\underline{n}=\max\{2\pi|n|,1\}$.
Then, we define a function space $\mathscr{X}_{\epsilon,s}$ as a completion of $\mathcal{T}$ with respect to $(\epsilon,s)$-norm $\|\cdot\|_{\epsilon,s}$, which forms a Hilbert space.
For a smooth Jordan curve $\Gamma$ with an $S^1$-parameterization $\Phi\colon S^1\rightarrow\mathbb{C}$, define a function space $\mathscr{X}_{\epsilon,s}(\Gamma)$ as $\mathscr{X}_{\epsilon,s}(\Gamma):=\{G\colon\Gamma\rightarrow\mathbb{C}\mid G\circ\Phi\in\mathscr{X}_{\epsilon,s}\}$, and the norm $\|G\|_{\mathscr{X}_{\epsilon,s}(\Gamma)}$ for $G\in\mathscr{X}_{\epsilon,s}(\Gamma)$ is given by $\|G\|_{\mathscr{X}_{\epsilon,s}(\Gamma)}:=\|G\circ\Phi\|_{\epsilon,s}$.
Moreover, we define a relation $(\epsilon_1,s_1)>(\epsilon_2,s_2)$ on $(0,+\infty)\times\mathbb{R}$ as $(\epsilon_1>\epsilon_2)\lor(\epsilon_1=\epsilon_2\land s_1>s_2)$.
For more details on these function spaces, see \cite{MR717692}.

Concerning arrangements of the singular points, the collocation points, and the dipole moments, we use a peripheral conformal mapping of $\partial\Omega$, which is introduced first in \cite{MR1362387,katsurada1998mathematical} and has been developed in \cite{MR3576615}.

\begin{dfn}
	For a Jordan curve $\Gamma$ in the plane and a positive constant $\rho$, a mapping $\Psi$ from a neighborhood of the circle $\gamma_\rho$ to $\mathbb{C}$ is called a \textit{peripheral conformal mapping} of $\Gamma$ with reference radius $\rho$ if the following two conditions are satisfied:
	\begin{enumerate}
		\item $\Psi$ maps $\gamma_\rho$ onto $\Gamma$;
		\item $\Psi\colon\overline{\mathcal{A}}_{\kappa^{-1}\rho,\kappa\rho}\rightarrow\mathbb{C}$ is a conformal mapping with some $\kappa>1$, where $\mathcal{A}_{\rho_1,\rho_2}$ represents an annular region $\{z\in\mathbb{C}\mid\rho_1<|z|<\rho_2\}$ with $\rho_2>\rho_1>0$.
	\end{enumerate}
\end{dfn}

\begin{rem}
	\label{rem:peripheral}
	The existence of peripheral conformal mapping is pointed out in \cite[Remark 3.1]{katsurada1998mathematical} without proofs, and its approximately construction based on FFT is explained in \cite{MR1362387}.
\end{rem}

Using these notions, we can state the following theorem proved in \cite{MR3576615}.

\begin{thm}[{\cite[Theorem 4.1]{MR3576615}}]
	\label{thm:DSM}
	Let us consider the following Dirichlet boundary value problem:
	\begin{equation}
		\begin{dcases*}
			\triangle u=0&in $\Omega$,\\
			u=b&on $\partial\Omega$.
		\end{dcases*}
		\label{eq:test_potential_problem}
	\end{equation}
	Suppose that $\partial\Omega$ is analytic, $b\in\mathscr{X}_{\mu,\sigma}(\partial\Omega)$ for some $(\mu,\sigma)>(1,1/2)$, $(\delta,t)\in(0,+\infty)\times\mathbb{R}$ satisfies 
	\begin{align}
		&1\le\delta\le\min\left\{\mu,\left(\frac{R}\rho\right)^2\right\},\\
		&\text{if $\delta=1$ then $t>1/2$ and $s<t$; if $\delta=\kappa$ then $t<-1/2$; if $\delta=\mu$ then $t\le\sigma$}.
	\end{align}
	Let $\Psi$ be a peripheral conformal mapping of $\partial\Omega$ with reference radius $\rho>0$, and arrange the singular points $\{\zeta_k\}_{k=1}^{N}$, the collocation points $\{z_j\}_{j=1}^{N}$, and the dipole moments $\{n_k\}_{k=1}^{N}$ as 
	\begin{align}
		\zeta_k
		&=
		\Psi(R\omega^k),
		\quad
		z_k
		=
		\Psi(\rho\omega^k),
		\quad
		n_k=\frac{\omega^k\Psi'(R\omega^k)}{|\Psi'(R\omega^k)|},
		\quad
		k=1,2,\ldots,N,
	\end{align}
	where $R\in(\rho,\sqrt{\kappa}\rho)$.
	Then, for sufficiently large $N$, the following holds:
	\begin{enumerate}
		\item There uniquely exists an approximate solution $u^{(N)}$ of the form \eqref{eq:DSM_approximate_solution} satisfying the collocation equations:
			\begin{equation}
				u^{(N)}(z_j)=b(z_j),
				\quad
				j=1,2,\ldots,N.
			\end{equation}
		\item There exist a positive constant $C=C(\epsilon,s,\delta,t,\rho,R)$ and a real constant $P=P(\epsilon,s,\delta,t)$ such that the error estimate
			\begin{equation}
				\|u-u^{(N)}\|_{H^s(\partial\Omega)}
				\le
				CN^P\frac{1}{\delta^{N/2}}\|b\|_{\mathscr{X}_{\delta,t}(\partial\Omega)}.
			\end{equation}
			Namely, an approximation error decays exponentially with respect to $N$ when the boundary data $b$ is analytic, but it decays algebraically with respect to $N$ when $b$ is not analytic but in $H^\sigma(\partial\Omega)$ with $\sigma>1/2$ (especially, $b$ is H\"older continuous).
	\end{enumerate}
\end{thm}

Applying the above theorem to the current situation, we can obtain the error estimate $\|g-g^{(N)}\|_{H^s(\partial\Omega)}$ for approximating the real part $g$.
However, we also have to establish the error estimate $\|h-h^{(N)}\|_{H^s(\partial\Omega)}$ for approximating the imaginary part $h$ in order to derive the approximation error for numerical conformal mapping.
To this end, we introduce the notion of a conjugate periodic function.

\begin{dfn}
	Let $\phi$ be a $1$-periodic function which has an absolutely convergent Fourier series expansion:
	\begin{equation}
		\phi(\tau)=\sum_{n=-\infty}^\infty a_n\mathrm{e}^{2\pi\mathrm{i}n\tau},
		\quad
		\tau\in S^1.
	\end{equation}
	Then, the \textit{conjugate periodic function} $\psi$ of $\phi$ is defined as
	\begin{equation}
		\psi(\tau)
		=
		\mathcal{H}\phi(\tau)
		=
		-\mathrm{i}\sum_{n=-\infty}^\infty\sigma_na_n\mathrm{e}^{2\pi\mathrm{i}n\tau},
		\quad
		\tau\in S^1.
	\end{equation}
	Here, $\mathcal{H}$ represents the Hilbert transform, and $\sigma_n$ denotes the signum function which is defined as 
	\begin{equation}
		\sigma_n=
		\begin{dcases}
			1,&n>0,\\
			0,&n=0,\\
			-1,&n<0.
		\end{dcases}
	\end{equation}
\end{dfn}

Using the notion of a conjugate periodic function, we derive the following error bound for the imaginary part.

\begin{thm}
	There exists a positive constant $C$ such that the following inequality holds:
	\begin{equation}
		\|h-h^{(N)}\|_{H^s(\partial\Omega)}
		\le
		C\|g-g^{(N)}\|_{H^s(\partial\Omega)}.
	\end{equation}
	Namely, the behavior of the approximation error for the imaginary part is governed by that for the real part.
\end{thm}

\begin{proof}
	Since the norm $\|\cdot\|_{\mathscr{X}_{\epsilon,s}(\partial\Omega)}$ is defined through a peripheral conformal mapping, and a potential problem on $\Omega$ can be converted into the one on the disk $D_\rho$, we only consider the case where $\Omega=D_\rho$.
	
	Denote the boundary values of $g$ and $g^{(N)}$ at $z=\rho\mathrm{e}^{2\pi\mathrm{i}\tau}$ as $G(\tau)$ and $G^{(N)}(\tau)$, respectively:
	\begin{equation}
		G(\tau):=g(\rho\mathrm{e}^{2\pi\mathrm{i}\tau}),
		\quad
		G^{(N)}(\tau):=g^{(N)}(\rho\mathrm{e}^{2\pi\mathrm{i}\tau}),
		\quad
		\tau\in S^1.
	\end{equation}
	Then, the approximation error $\|g-g^{(N)}\|_{\mathscr{X}_{\epsilon,s}(\partial\Omega)}$ can be represented by using $G$ and $G^{(N)}$ as follows:
	\begin{align}
		\|g-g^{(N)}\|_{H^s(\partial\Omega)}^2
		&=
		\|G-G^{(N)}\|_{H^s}^2\\
		&=
		|\hat{G}(0)-\hat{G}^{(N)}(0)|^2+\sum_{n\in\mathbb{Z}\setminus\{0\}}|\hat{G}(n)-\hat{G}^{(N)}(n)|^2\underline{n}^{2s}.
	\end{align}
	Defining functions $H$, $\tilde{H}$, $H^{(N)}$, $\tilde{H}^{(N)}$ as
	\begin{align}
		&H(\tau):=h(\rho\mathrm{e}^{2\pi\mathrm{i}\tau}),
		\quad
		\tilde{H}(\tau):=\tilde{h}(\rho\mathrm{e}^{2\pi\mathrm{i}\tau}),\\
		&H^{(N)}(\tau):=h^{(N)}(\rho\mathrm{e}^{2\pi\mathrm{i}\tau}),
		\quad
		\tilde{H}^{(N)}(\tau):=\tilde{h}^{(N)}(\rho\mathrm{e}^{2\pi\mathrm{i}\tau})
	\end{align}
	for $\tau\in S^1$, we have the following relations:
	\begin{align}
		&\tilde{H}(\tau)=\mathcal{H}G(\tau),
		\quad
		\tilde{H}^{(N)}(\tau)=\mathcal{H}G^{(N)}(\tau),\\
		&H(\tau)=\tilde{H}(\tau)-h^{(N)}(z_0),
		\quad
		H^{(N)}(\tau)=\tilde{H}^{(N)}(\tau)-h^{(N)}(z_0),
	\end{align}
	that is,
	\begin{align}
		&\tilde{H}(\tau)
		=
		-\mathrm{i}\sum_{n\in\mathbb{Z}}\sigma_n\hat{G}(n)\mathrm{e}^{2\pi\mathrm{i}n\tau},
		\quad
		\tilde{H}^{(N)}(\tau)
		=
		-\mathrm{i}\sum_{n\in\mathbb{Z}}\sigma_n\hat{G}^{(N)}(n)\mathrm{e}^{2\pi\mathrm{i}n\tau}
	\end{align}
	for $\tau\in S^1$, which yields that
	\begin{align}
		&\tilde{H}(\tau)-\tilde{H}^{(N)}(\tau)
		=
		-\mathrm{i}\sum_{n\in\mathbb{Z}}\sigma_n(\hat{G}(n)-\hat{G}^{(N)}(n))\mathrm{e}^{2\pi\mathrm{i}n\tau},
		\\
		&H(\tau)-H^{(N)}(\tau)
		=
		-(\tilde{h}(z_0)-\tilde{h}^{(N)}(z_0))-\mathrm{i}\sum_{n\in\mathbb{Z}}\sigma_n(\hat{G}(n)-\hat{G}^{(N)}(n))\mathrm{e}^{2\pi\mathrm{i}n\tau}.
	\end{align}
	The approximation error for $\tilde{h}$ can be estimated as
	\begin{align}
		\|\tilde{h}-\tilde{h}^{(N)}\|_{H^s(\partial\Omega)}
		&=
		\|\tilde{H}-\tilde{H}^{(N)}\|_{H^s}
		=
		\sum_{n\in\mathbb{Z}}|\sigma_n(\hat{G}(n)-\hat{G}^{(N)}(n))|^2\underline{n}^{2s}\\
		&=
		\sum_{n\in\mathbb{Z}\setminus\{0\}}|\hat{G}(n)-\hat{G}^{(N)}(n)|^2\underline{n}^{2s}
		\le
		\|G-G^{(N)}\|_{H^s}^2\\
		&=
		\|g-g^{(N)}\|_{H^s(\partial\Omega)}^2.
	\end{align}
	By continuous embedding $H^s(\partial\Omega)\subset C(\partial\Omega)$ for $s>1/2$, there exists a positive constant $C$ such that $\|v\|_{L^\infty(\Gamma)}\le C\|v\|_{H^s(\partial\Omega)}$ holds for all $v\in H^s(\partial\Omega)$.
	Therefore, we obtain the following estimate via the maximum principle for harmonic functions:
	\begin{equation}
		|\tilde{h}(z_0)-\tilde{h}^{(N)}(z_0)|^2
		\le
		\|\tilde{h}-\tilde{h}^{(N)}\|_{L^\infty(\Gamma)}^2
		\le
		C\|\tilde{h}-\tilde{h}^{(N)}\|_{H^s(\Gamma)}^2.
	\end{equation}
	Summarizing the above, we obtain
	\begin{align}
		\|h-h^{(N)}\|_{H^s(\partial\Omega)}
		&=
		\|H-H^{(N)}\|_{H^s(\partial\Omega)}
		\le
		C\|g-g^{(N)}\|_{H^s(\partial\Omega)},
	\end{align}
	which is the desired estimate.
\end{proof}

\section{Numerical conformal mapping for multiply-connected region}
\label{sec:MultiplyConnected}

In this section, we extend our method developed in the previous section to the case where $\Omega$ is a multiply-connected region.

\subsection{Case 1: Doubly-connected region}

We first consider the case where $\Omega$ is a nondegenerate doubly-connected region, that is, the complement of $\Omega$ on the Riemann sphere has the components $K_0$ and $K_1$, where $K_0$ is the unbounded component in a sense that $K_0$ contains $\infty$, and neither $K_0$ nor $K_1$ reduces to a single point.
In this case, it is well known that the following mapping theorem holds.

\begin{thm}[{\cite[Theorem 17.1a]{MR822470}}]
	Let $\Omega$ be a nondegenerate doubly-connected region.
	Then, there exists a unique $\rho\in(0,1)$ such that there exists a one-to-one analytic function $f$ that maps $\Omega$ onto the annular region $A_{\rho,1}$.
	If the outer boundaries correspond to each other, then $f$ is determined up to a rotation of the annular region.
\end{thm}

The number $\rho^{-1}$ is called the modulus of $\Omega$.
Based on Amano's method, in computing the forward conformal mapping $f$, we first arbitrarily take a base point $z_0\in K_1$, and express $f$ in the following form:
\begin{equation}
	f(z)=(z-z_0)\exp[g(z)+\mathrm{i}h(z)].
\end{equation}
Then, a function $g$ is characterized as the solution to the following potential problem:
\begin{equation}
	\begin{dcases*}
		\triangle g=0&in $\Omega$,\\
		g(z)=-\log|z-z_0|&for $z\in\partial K_0$,\\
		g(z)=\log\rho-\log|z-z_0|&for $z\in\partial K_1$.
	\end{dcases*}
\end{equation}
We here note that unknowns in the above problem are $g$ and $\rho$.
Also, note that the existence of a conjugate harmonic function $h$ of $g$ is not apparent in this case.
In order to describe the sufficient and necessary condition for the existence of the conjugate harmonic function, we introduce the notion of the conjugate period.

\begin{dfn}[{\cite{MR822470}}]
	Let $\Omega$ be a region of finite connectivity, that is, the complement of $\Omega$ on the Riemann sphere have the components $K_1,K_2,\ldots,K_n$, where $K_0$ is the possibly empty unbounded component.
	Let $\Gamma_i$ be a closed curve in $\Omega$ which has winding number $+1$ for all points $z\in K_i$, and which has winding number $0$ for all points $z\in K_j$ ($j\neq i$).
	For a differentiable function $u$ in $\Omega$, the quantity
	\begin{equation}
		\eta_i:=\int_{\Gamma_i}(-u_y\,\mathrm{d}x+u_x\,\mathrm{d}y)
	\end{equation}
	is called the conjugate period of $u$ with respect to $K_i$.
\end{dfn}

Using the conjugate period, we can write down the sufficient and necessary condition for the existence of the conjugate harmonic function as in the following theorem.

\begin{thm}[{\cite[Theorem 15.1d]{MR822470}}]
	A harmonic function $u$ in a finitely connected region $\Omega$ has a conjugate harmonic function in $\Omega$ if and only if the conjugate periods of $u$ for all connected components are equal to zero.
\end{thm}

If we apply the MFS to compute $g$, an approximate solution $g^{(N)}$ is of the form
\begin{equation}
	g^{(N)}(z)
	=
	\sum_{\nu=1}^2\sum_{k=1}^{N}Q_{\nu k}\log|z-\zeta_{\nu k}|,
\end{equation}
where $\{\zeta_{\nu k}\}_{k=1}^{N}\subset\mathring{K}_\nu$ ($\nu=1,2$) are the singular points.
The conjugate period of $g^{(N)}$ for $K_2$ can be computed as
\begin{equation}
	\int_{\Gamma_2}(-g_y^{(N)}\,\mathrm{d}x+g_x^{(N)}\,\mathrm{d}y)
	=
	2\pi\sum_{k=1}^{N}Q_{2k}.
\end{equation}
Therefore we have to add the following condition to assure the existence of the conjugate harmonic function $h^{(N)}$ of $g^{(N)}$:
\begin{equation}
	\sum_{k=0}^{N-1}Q_{2k}=0.
	\label{eq:condition_conjugate_period}
\end{equation}
Hence combining the above condition and the collocation equations
\begin{alignat}{2}
	&g^{(N)}(z_{0j})=-\log|z_{0j}-z_0|,&\quad&j=0,1,\ldots,N-1,\\
	&g^{(N)}(z_{1j})=\log R^{(N)}-\log|z_{1j}-z_0|,&\quad&j=0,1,\ldots,N-1,
\end{alignat}
we can compute $g^{(N)}$ together with an approximation $R^{(N)}$ of $\rho$, where $\{z_{\mu j}\}_{j=1}^{N}\subset\partial K_\mu$ ($\mu=1,2$) are the collocation points.

On the other hand, if we apply the DSM, $g^{(N)}$ is given by
\begin{equation}
	g^{(N)}(z)=\sum_{\nu=1}^2\sum_{k=1}^{N}Q_{\nu k}\Re\left(\frac{n_{\nu k}}{z-\zeta_{\nu k}}\right),
\end{equation}
and its conjugate period with respect to $K_2$ is equal to $0$.
Therefore, the existence of the conjugate harmonic function of $g^{(N)}$ is always assured without any condition such as \eqref{eq:condition_conjugate_period}, and it is given by
\begin{equation}
	h^{(N)}(z)=\sum_{\nu=1}^2\sum_{k=1}^{N}Q_{\nu k}\Im\left(\frac{n_{\nu k}}{z-\zeta_{\nu k}}\right).
\end{equation}
Hence there remains one degree of freedom to determine $\{Q_{\nu k}\}$ and $R^{(N)}$, which implies that there exists a possibility to construct more ``suitable'' numerical conformal mapping depending on the problem to be considered.
However, we here adopt the same condition \eqref{eq:condition_conjugate_period} for Amano's method to compare numerical results in the almost same situation.

\subsection{$n$-ly connected region ($n\ge3$)}

In this case, there are several possibilities for canonical regions.

\begin{thm}[{\cite{MR822470}}]
	Let $\Omega$ be a nondegenerate $n$-ly connected region with $n\ge3$, that is, the complement of $\Omega$ on the Riemann sphere have the components $K_1,K_2,\ldots,K_n$, where $K_1$ is the unbounded component.
	Then, there exist $n-1$ real numbers $\rho_j$ $(j=2,3,\ldots,n)$ such that $0<\rho_n<\rho_j<1$ $(j=2,3,\ldots,n-1)$, and that there exists an analytic function $f$ that maps $\Omega$ conformally onto the annular region $\mathcal{A}_{\rho_n,1}$ cut along $n-2$ mutually disjoint arcs $\Lambda_j$ located on the circles $|w|=\rho_j$ $(j=2,3,\ldots,n-1)$.
	The mapping function $f$ can be extended analytically to the curves $\partial K_j$.
	The images of $\partial K_1$ and of $\partial K_n$ are the circles $\Lambda_0\colon|w|=1$ and $\Lambda_n\colon|w|=\rho_{n}$, respectively.
	The images of the curves $\partial K_j$ are the arcs $\Lambda_j$ $(j=2,3,\ldots,n-1)$, traversed twice.
	The function $f$ is determined up to a factor of modulus $1$.
\end{thm}

Also, in this case, we represent $f$ as
\begin{equation}
	f(z)=(z-z_0)\exp[g(z)+\mathrm{i}h(z)],
\end{equation}
where $z_0$ is a base point arbitrarily chosen from $K_n$.
In Amano's method, an approximation $g^{(N)}$ of $g$ is given by
\begin{equation}
	g^{(N)}(z)=\sum_{\nu=1}^n\sum_{k=1}^{N}Q_{\nu k}\log|z-\zeta_{\nu k}|,
\end{equation}
where $\{\zeta_{\nu k}\}_{k=1}^{N}\subset\mathring{K}_\nu$ ($\nu=1,2,\ldots,n$) are the singular points, and its conjugate period with respect to $K_\nu$ is equal to $2\pi\sum_{k=1}^{N}Q_{\nu k}$ ($\nu=2,3,\ldots,n$).
Therefore coefficients $\{Q_{\nu k}\}$ and approximations $\{R_\mu\}$ of $\{\rho_\mu\}$ are obtained by solving the following collocation equations and conditions for the existence of the conjugate harmonic function:
\begin{alignat}{2}
	&g(z_{1j})=-\log|z_{1j}-z_0|&\quad&(j=1,2,\ldots,N),\\
	&g(z_{\mu j})=\log R_\mu^{(N)}-\log|z_{\mu j}-z_0|&\quad&(\mu=2,3,\ldots,n;\,j=1,2,\ldots,N),\\
	&\sum_{k=1}^{N}Q_{\nu k}=0&\quad&(\nu=2,3,\ldots,n)
\end{alignat}

If we apply our method, $g^{(N)}$ is given by
\begin{equation}
	g^{(N)}(z)=\sum_{\nu=0}^n\sum_{k=0}^{N-1}Q_{\nu k}\Re\left(\frac{n_{\nu k}}{z-\zeta_{\nu k}}\right),
\end{equation}
and conjugate periods of $g^{(N)}$ with respect to $K_\nu$ ($\nu=2,3,\ldots,n$) are all equal to zero; thus, we can freely add $n$ equations to determine $\{Q_{\nu k}\}$ and $\{R_\mu^{(N)}\}$.

%If we consider an unbounded multiply-connected region $\Omega$, there are 

\section{Numerical experiments}
\label{sec:numerical_experiments}

We show the results of several numerical experiments in this section.

\subsection{On the arrangements of the singular points, the collocation points, and the dipole moments}
\label{subsec:Arrangements}

In Theorem \ref{thm:DSM}, we gave an example on the arrangements of the singular points, the collocation points, and the dipole moments by using the peripheral conformal mapping.
As we have mentioned in Remark \ref{rem:peripheral}, the peripheral conformal mapping can be approximately constructed by using FFT; therefore, the arrangements by the peripheral conformal mapping would produce ``nice'' numerical results.
On the other hand, in a series of the study on numerical conformal mappings by Amano's group, the following simple arrangements of the singular points $\{\zeta_k^{\mathrm{A}}(r)\}_{k=1}^N$, which we call Amano's arrangement, are adopted:
\begin{equation}
	\zeta_k^{\mathrm{A}}(r)
	=
	z_k-\frac{\mathrm{i}r}{2}(z_{k+1}-z_{k-1}),
	\quad
	k=1,2,\ldots,N,
	\quad
	r>0,
\end{equation}
where $z_0:=z_N$ and $z_{N+1}:=z_1$.
Roughly speaking, the above rule places the singular points in approximately outward normal directions (see Figure \ref{fig:AmanoRule} for geometrical interpretation).
\begin{figure}[tb]
	\centering
	\includegraphics[width=.5\hsize]{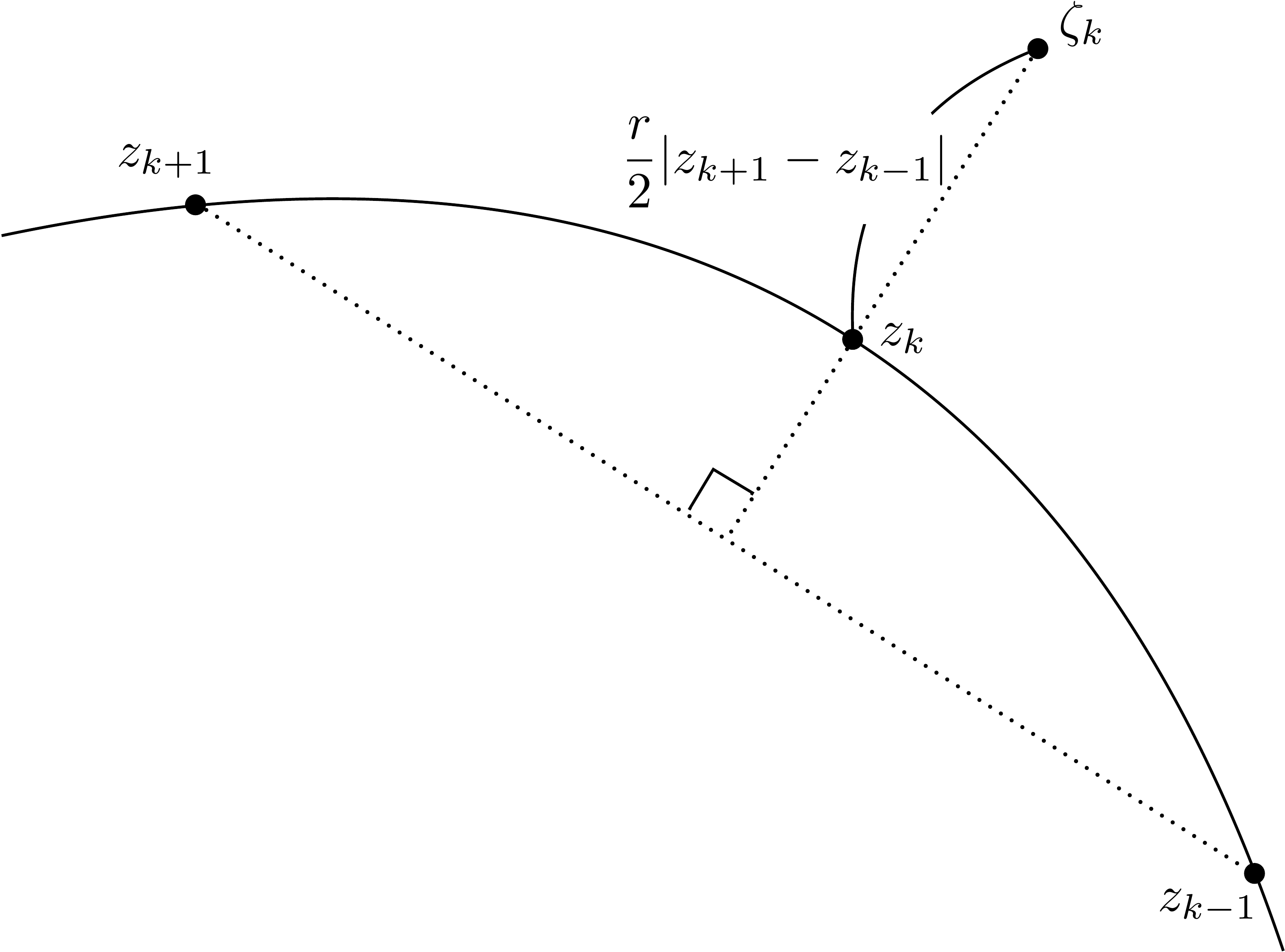}
	\caption{Amano's arrangement of the singular points}
	\label{fig:AmanoRule}
\end{figure}
The rule of Amano's arrangements is quite simple and its implementation is so easy; however, it is numerically observed that they offer highly accurate numerical results.
Therefore, it is natural to expect that there are some relationships between the arrangements by using conformal mappings and Amano's arrangements.
The following theorem gives a partial answer to the above conjecture.

\begin{thm}
	Amano's arrangement is a linear approximation of the arrangement by conformal mapping.
	More precisely, the following relation holds:
	\begin{equation}
		\zeta_k^{\mathrm{C}}(R)
		=
		\zeta_k^{\mathrm{A}}\left(\frac{R-1}{\sin(2\pi/N)}\right)+O((R-1)^2)+O((R-1)N^{-1})
		\quad
		\text{as}\ R\downarrow1,\ N\to\infty,
	\end{equation}
	where $\zeta_k^{\mathrm{C}}(R)=\Psi(R\omega^{k})$ represents the singular points arranged by the peripheral conformal mapping $\Psi$ of $\partial\Omega$ with reference radius $1$.
\end{thm}

\begin{proof}
	The key and only tool in the proof is the Taylor expansion.
	By the Taylor expansion, we have
	\begin{align}
		\zeta_k^{\mathrm{C}}(R)
		&=
		\Psi(R\omega^k)
		=
		\Psi(\omega^k+(R-1)\omega^k)\\
		&=
		\Psi(\omega^k)+\Psi'(\omega^k)(R-1)\omega^k+O((R-1)^2)\\
		&=
		z_k+(R-1)\Psi'(\omega^k)\omega^k+O((R-1)^2)
		\label{eq:Taylor}
	\end{align}
	as $R\downarrow1$.
	Moreover, we obtain the following relations by using the Taylor expansion again:
	\begin{align}
		z_{k+1}
		&=
		\Psi(\omega^{k+1})
		=
		\Psi(\omega^k+(\omega^{k+1}-\omega^k))\\
		&=
		\Psi(\omega^k)+\Psi'(\omega^k)(\omega^{k+1}-\omega^k)+O(|\omega-1|^2)\\
		&=
		z_k+\Psi'(\omega^k)\omega^k(\omega-1)+O(N^{-2}),
	\end{align}
	and
	\begin{align}
		z_{k-1}
		=
		z_k+\Psi'(\omega^k)\omega^k(\omega^{-1}-1)+O(N^{-2})
	\end{align}
	as $N\to\infty$.
	Combining the above two relations, we have
	\begin{equation}
		\Psi'(\omega^k)\omega^k
		=
		-\frac{\mathrm{i}}{2\sin(2\pi/N)}(z_{k+1}-z_{k-1})
		+
		O(N^{-1})
		\quad
		\text{as}\ N\to\infty.
	\end{equation}
	Substituting the above expression into \eqref{eq:Taylor}, we obtain
	\begin{align}
		\zeta_k^{\mathrm{C}}(R)
		&=
		z_k-\frac{\mathrm{i}(R-1)}{2\sin(2\pi/N)}(z_{k+1}-z_{k-1})+O((R-1)^2)+O((R-1)N^{-1})\\
		&=
		\zeta_k^{\mathrm{A}}\left(\frac{R-1}{\sin(2\pi/N)}\right)+O((R-1)^2)+O((R-1)N^{-1})
	\end{align}
	as $R\downarrow1$ and $N\to\infty$, which is the desired relation.
\end{proof}

There are several possibilities on the choices of the singular points $\{\zeta_k\}_{k=1}^{N}$, the collocation points $\{z_j\}_{j=1}^{N}$, and the dipole moments $\{n_k\}_{k=1}^{N}$ in Algorithm \ref{alg:forward_numerical_conformal_mapping}, and the singular points $\{\xi_k\}_{k=1}^{N}$ in Algorithm \ref{alg:backward_numerical_conformal_mapping}.
In the following numerical examples, they are given by the rules below:
The collocation points $\{z_j\}_{j=1}^N$ are given by $z_j=\Phi(j/N)$, where $\Phi\colon[0,1]\rightarrow\partial\Omega$ is an $S^1$-parameterization of $\partial\Omega$.
The singular points $\{\zeta_k\}_{k=1}^{N}$ for forward conformal mapping and $\{\xi_k\}_{k=1}^{N}$ for backward conformal mapping are given by Amano's arrangement, that is,
\begin{equation}
	\zeta_k
	=
	\zeta_k^{\mathrm{A}}(r_{\mathrm{f}})
	=
	z_k-\frac{\mathrm{i}r_{\mathrm{f}}}{2}(z_{k+1}-z_{k-1}),
	\quad
	\xi_k
	=
	\xi_k^{\mathrm{A}}(r_{\mathrm{b}})
	=
	w_k-\frac{\mathrm{i}r_{\mathrm{b}}}{2}(w_{k+1}-w_{k-1}),
\end{equation}
for $k=1,2,\ldots,N$.
The dipole moments $\{n_k\}_{k=1}^{N}$ are given by Amano-like rule.
Namely, we define $n_k$ as 
\begin{equation}
	n_k=-\mathrm{i}\frac{\zeta_{k+1}-\zeta_{k-1}}{|\zeta_{k+1}-\zeta_{k-1}|},
	\quad
	k=1,2,\ldots,N.
\end{equation}
The parameters $r_{\mathrm{f}}$ and $r_{\mathrm{b}}$ for arranging the singular points are given in the form $r_{\mathrm{f}}=\tilde{r}_{\mathrm{f}}N$ and $r_{\mathrm{b}}=\tilde{r}_{\mathrm{b}}N$, respectively, because $|z_{k+1}-z_{k-1}|$ and $|w_{k+1}-w_{k-1}|$ are approximately equal to $O(N^{-1})$.
The explicit values of $\tilde{r}_{\mathrm{f}}$ and $\tilde{r}_{\mathrm{b}}$, and the position of the base point $z_0$ are written in each numerical experiments.

\subsection{Simply-connected region}

\subsubsection{Disk--Disk}

The first example is the simplest one, a bidirectional numerical conformal mapping between the unit disk.
In this case, the forward numerical conformal mapping $f$ and backward one $f_*$ are explicitly given by
\begin{equation}
	f(z)=\frac{z-z_0}{1-\overline{z}_0z},
	\quad
	f_*(w)=\frac{w+z_0}{1+\overline{z}_0w}.
\end{equation}
Figure \ref{fig:DiskToDiskForward} depicts results of our numerical experiments for forward conformal mapping.
\begin{figure}[tb]
	\begin{minipage}{.435\hsize}
		\includegraphics[width=\hsize]{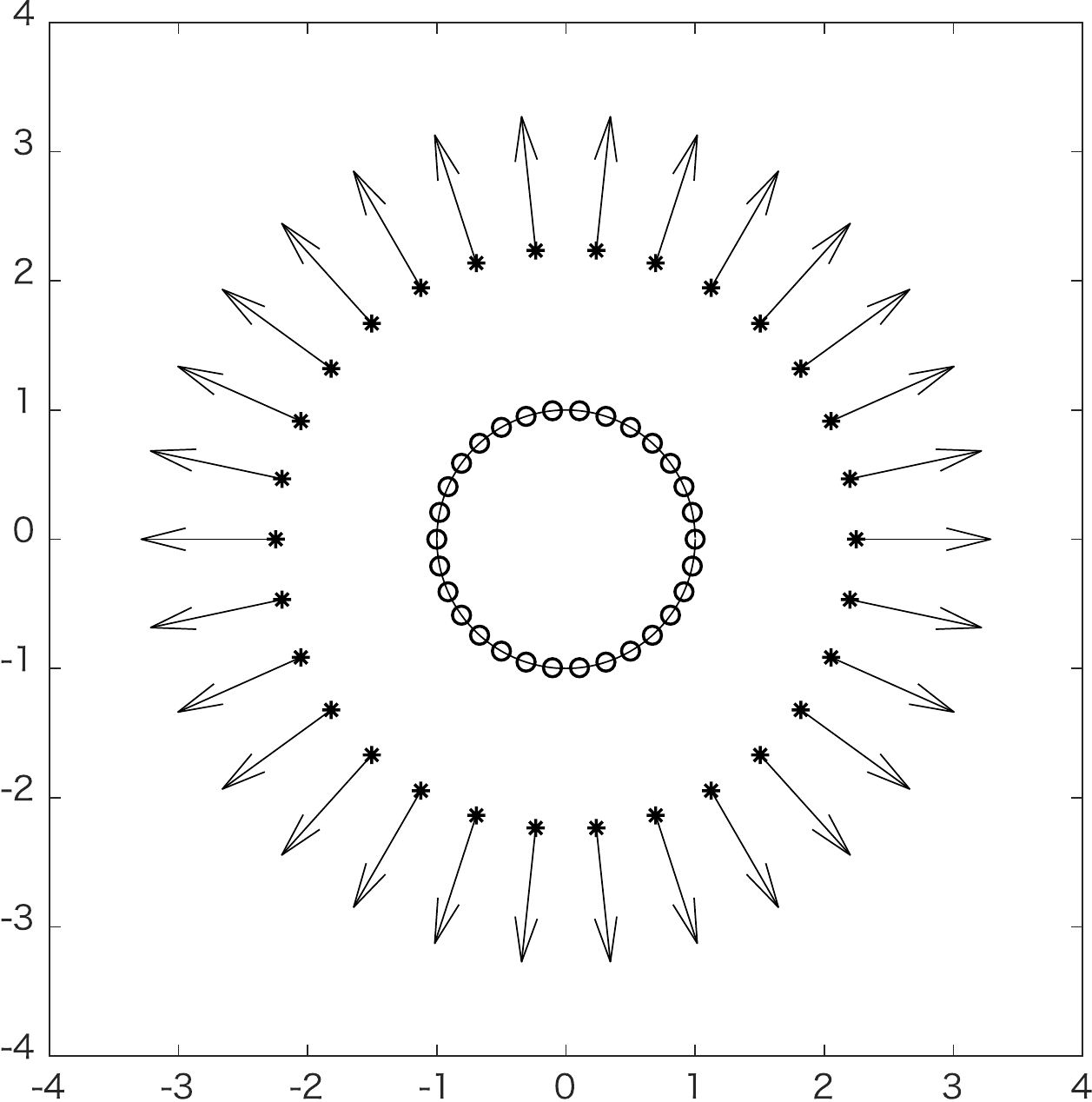}
	\end{minipage}%
	\begin{minipage}{.565\hsize}
		\includegraphics[width=\hsize]{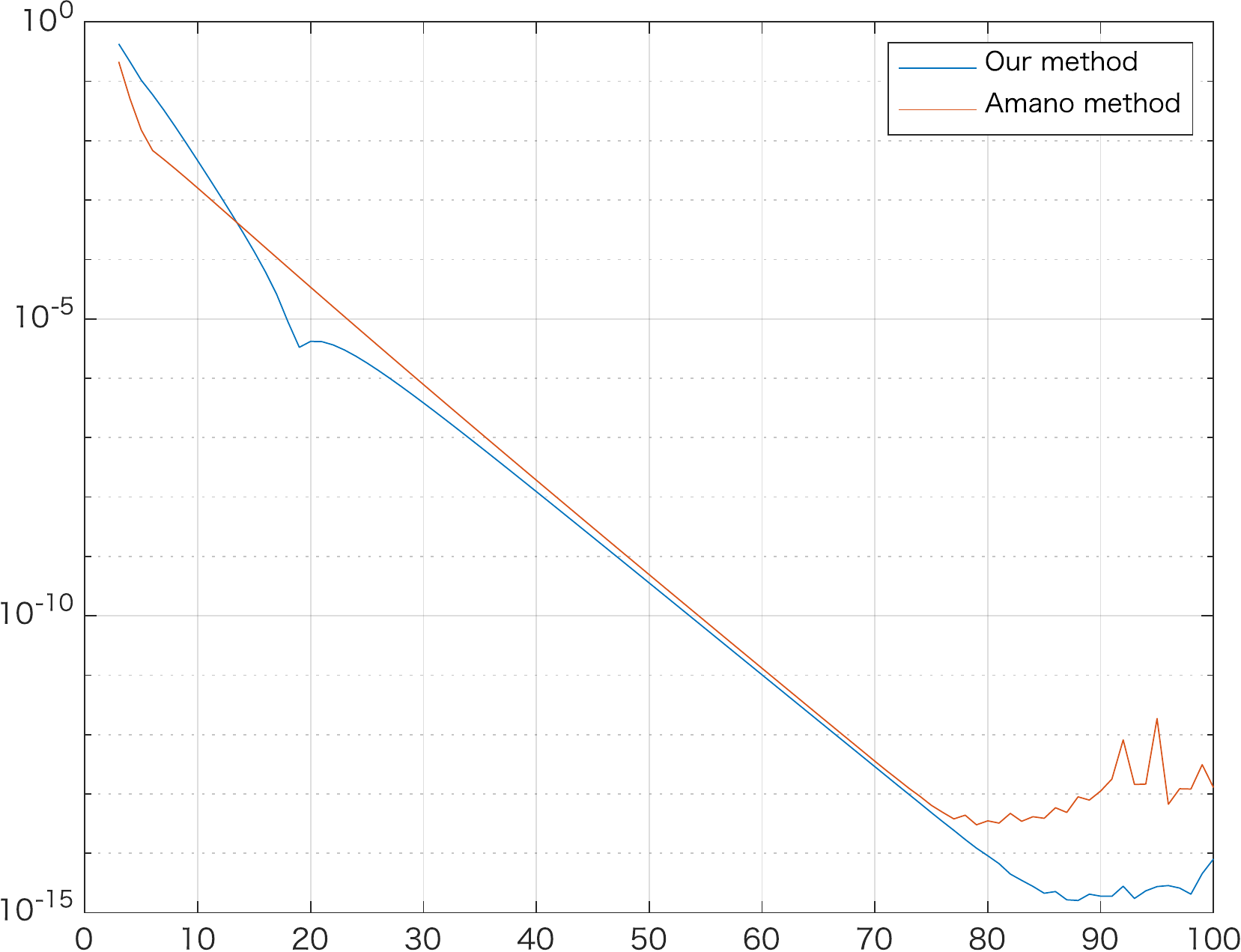}
	\end{minipage}%
	
	\begin{minipage}{.5\hsize}
		\includegraphics[width=\hsize]{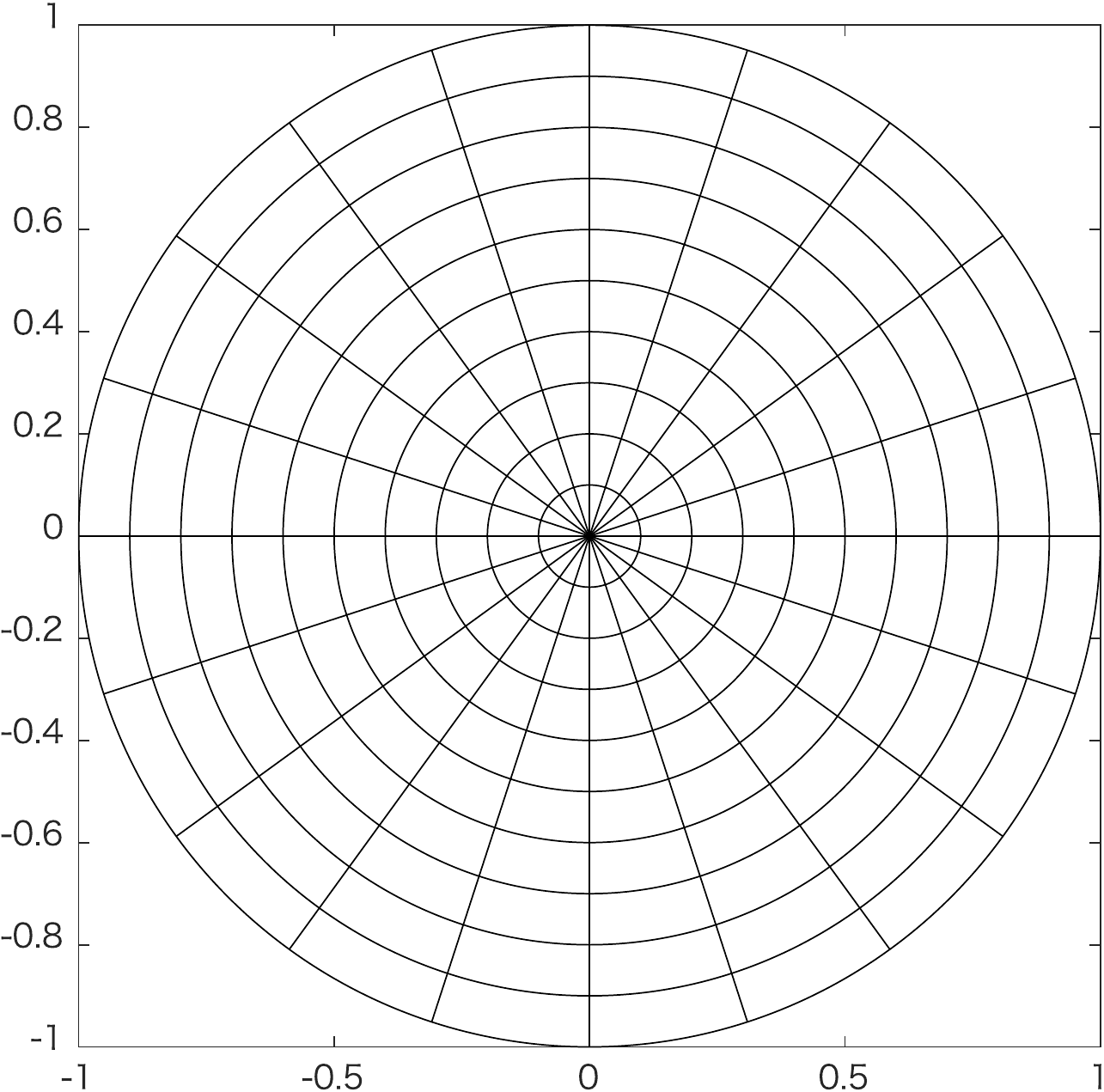}
	\end{minipage}%
	\begin{minipage}{.5\hsize}
		\includegraphics[width=\hsize]{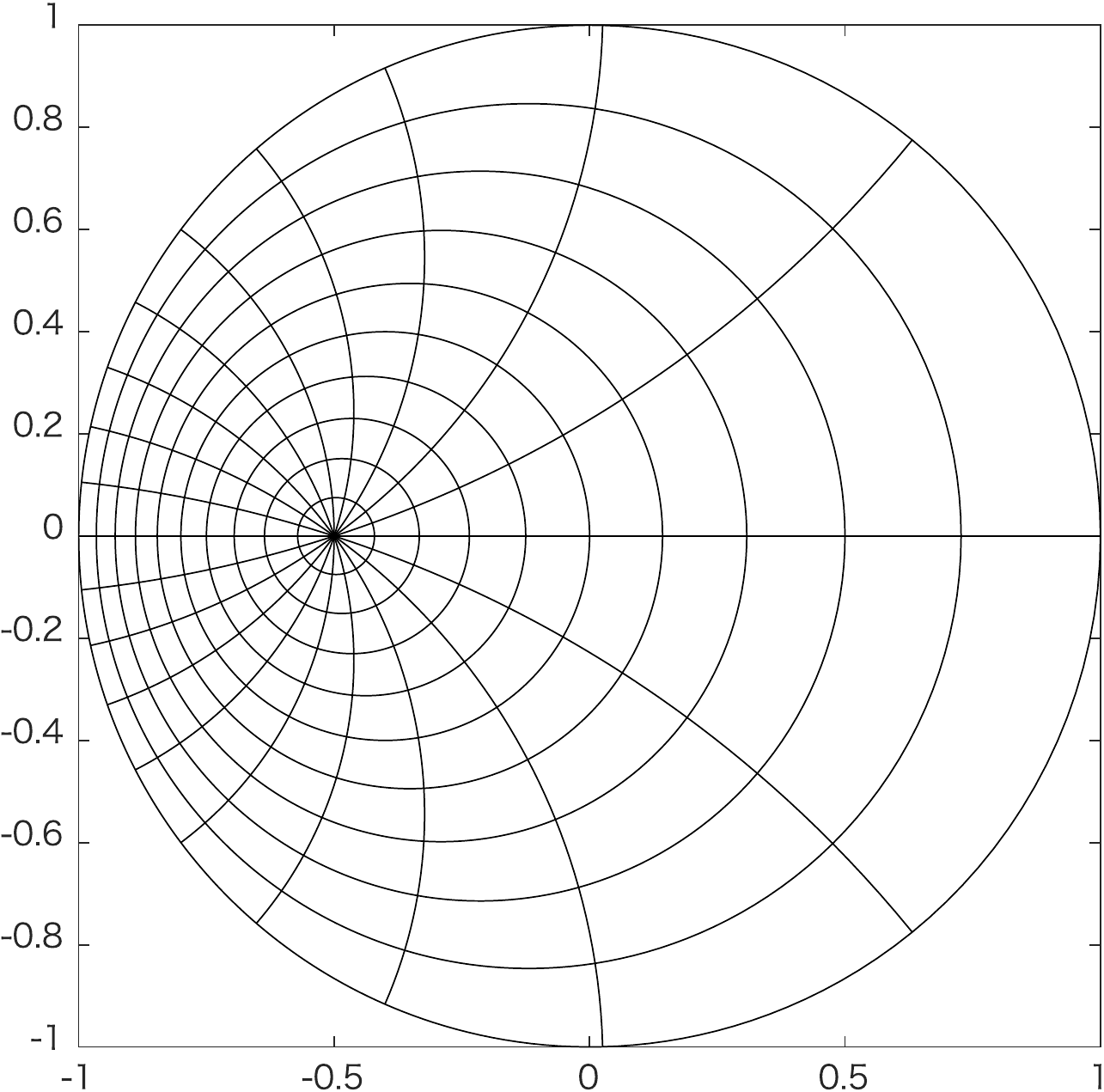}
	\end{minipage}%
	\caption{Numerical results for forward conformal mapping from the unit disk onto itself: (upper left) the arrangements of the singular points $\{\zeta_k\}_{k=1}^{N}$, the collocation points $\{z_j\}_{j=1}^{N}$, and the dipole moments $\{n_k\}_{k=1}^{N}$; (upper right) $N$-$\log_{10}\|f-f^{(N)}\|_{L^\infty(\Omega)}$ graph; (lower left) preimage; (lower right) image by numerical forward conformal mapping, where $z_0=0.5$, and $r_{\mathrm{f}}=0.2$.}
	\label{fig:DiskToDiskForward}
\end{figure}
The upper left figure depicts the arrangements of the singular points $\{\zeta_k\}_{k=1}^{N}$, the collocation points $\{z_j\}_{j=1}^{N}$, and the dipole moments $\{n_k\}_{k=1}^{N}$ when $N=30$.
The upper right figure represents the behavior of the approximation error $\|u-u^{(N)}\|_{L^\infty(\Omega)}$.
The horizontal axis represents the number $N$ of points, and the vertical axis the common logarithm of the approximation error.
We compare our method with Amano's method and can observe that the approximation errors for both methods decay exponentially for $N$.
The lower left figure shows the preimage, and its image by the numerical conformal mapping is depicted in the lower right figure.
Numerical results for backward conformal mapping are shown in Figure \ref{fig:DiskToDiskBackward}.
The parameters are given as $z_0=0.5$, $r_{\mathrm{f}}=0.2$, and $r_{\mathrm{b}}=0.1$.
\begin{figure}[tb]
	\begin{minipage}{.44\hsize}
		\centering
		\includegraphics[width=\hsize]{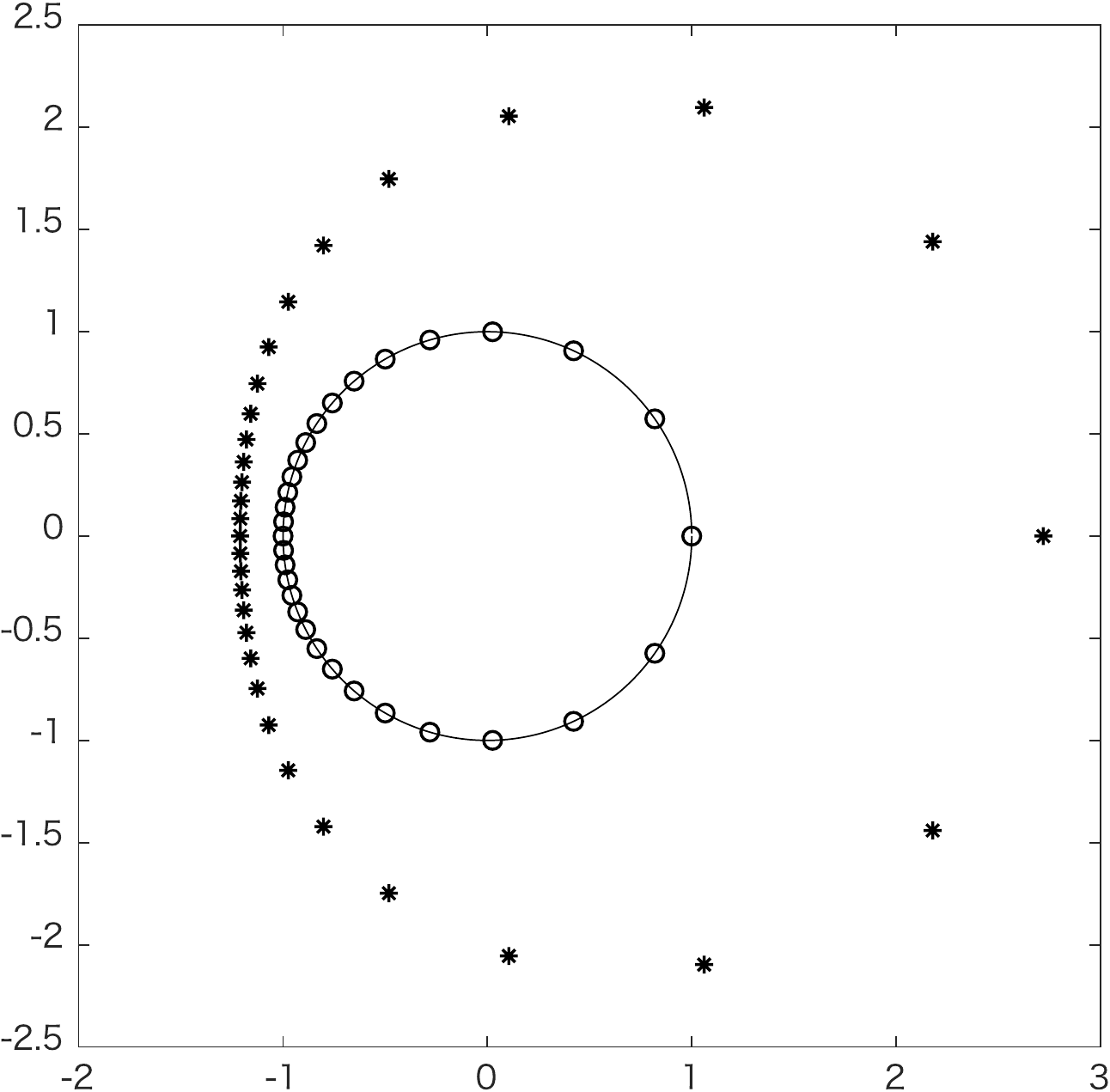}
	\end{minipage}%
	\begin{minipage}{.56\hsize}
		\centering
		\includegraphics[width=\hsize]{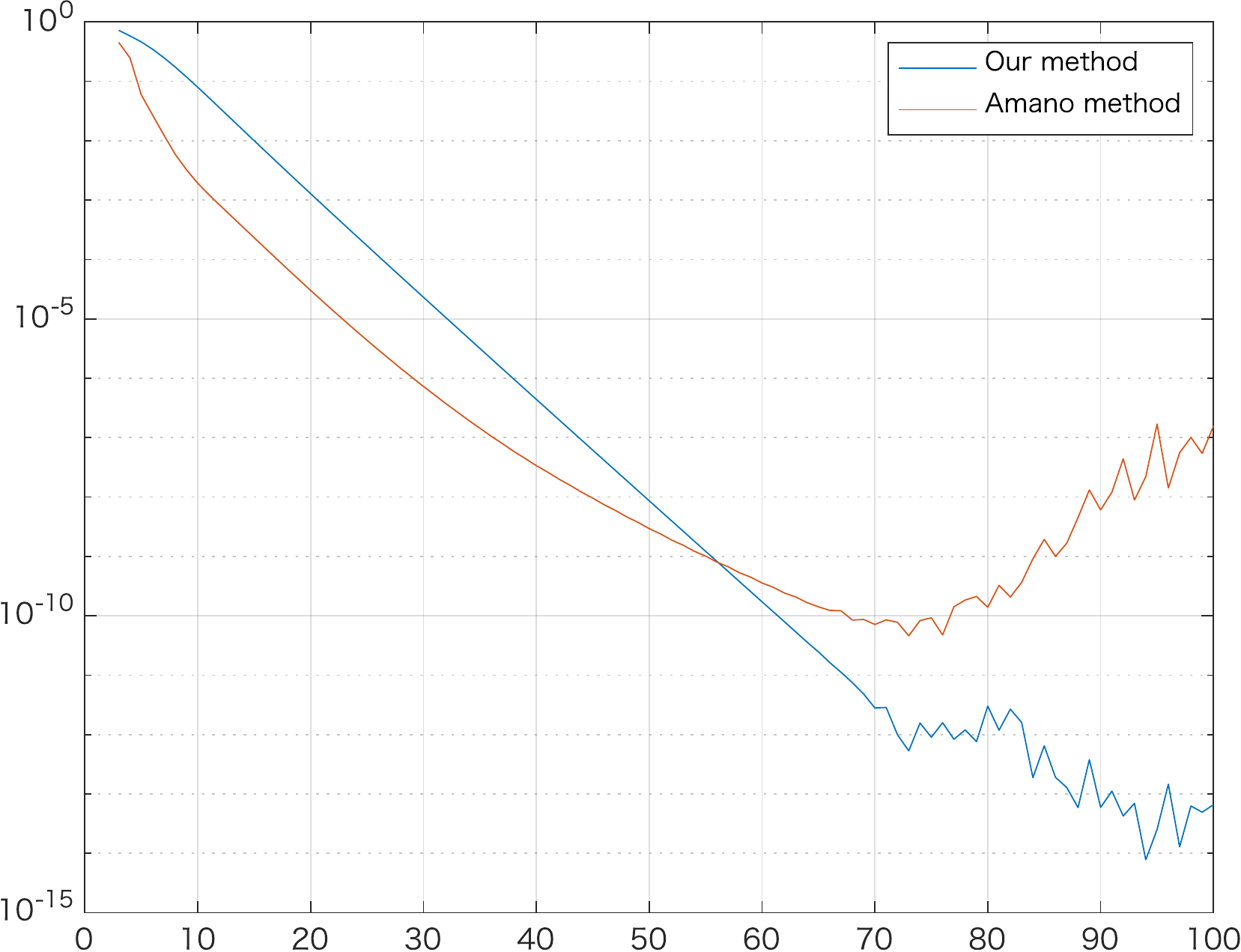}
	\end{minipage}%
	
	\begin{minipage}{.5\hsize}
		\includegraphics[width=\hsize]{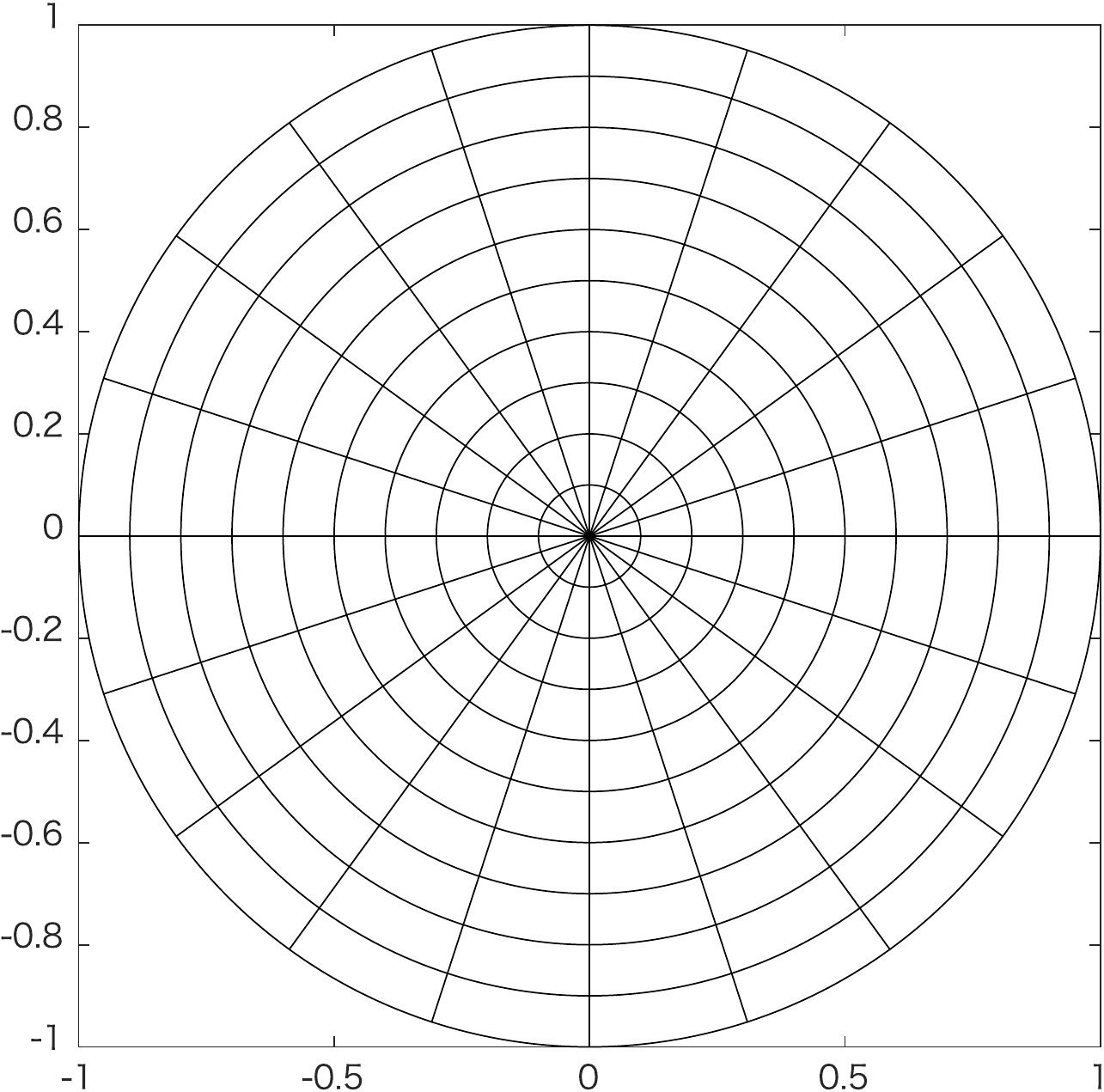}
	\end{minipage}%
	\begin{minipage}{.5\hsize}
		\includegraphics[width=\hsize]{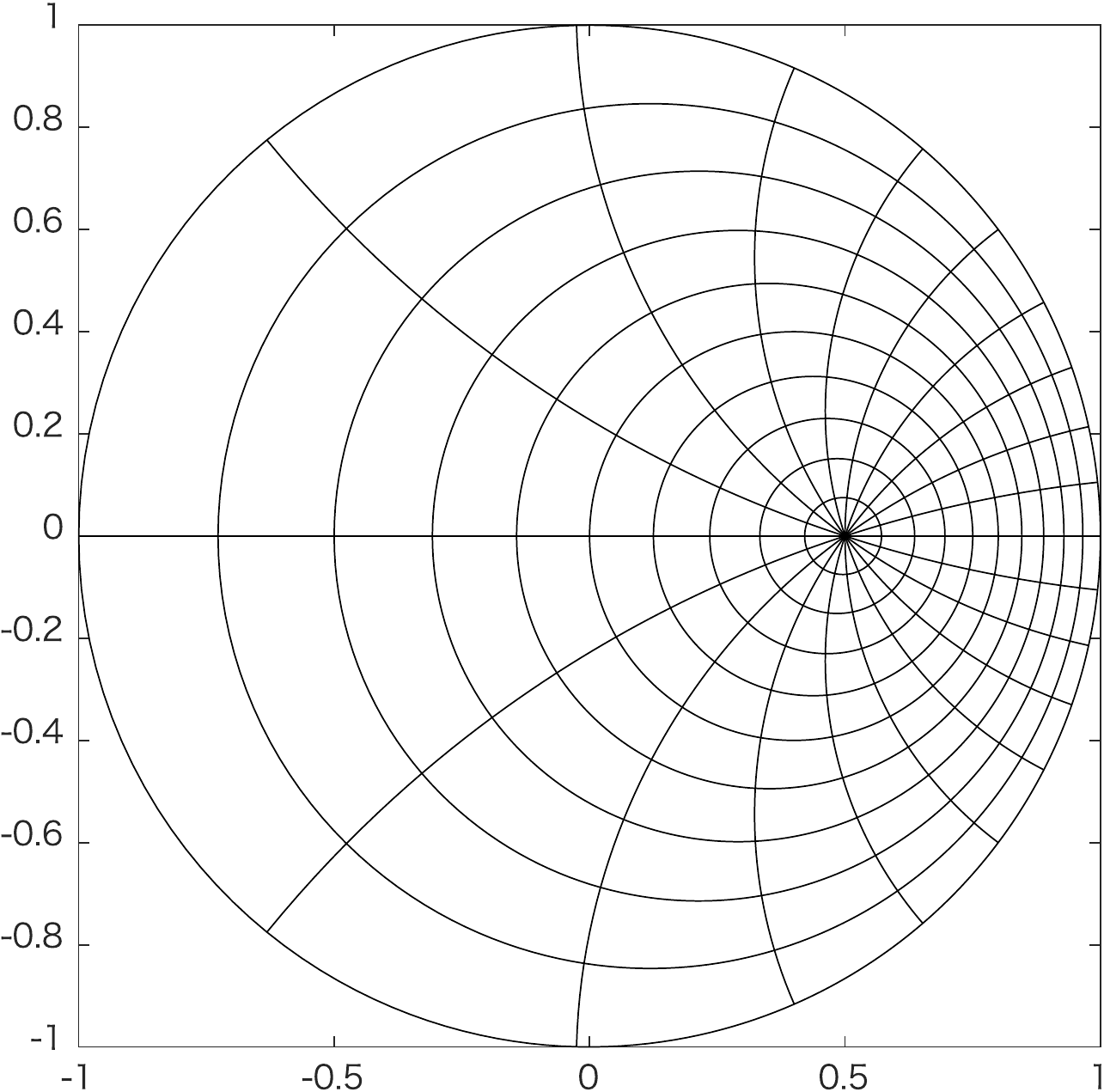}
	\end{minipage}%
	\caption{Numerical results for backward conformal mapping from the unit disk onto itself: (upper left) the arrangements of the singular points $\{\xi_k\}_{k=1}^{N}$, and the collocation points $\{w_j\}_{j=1}^{N}$; (upper right) $N$-$\log_{10}\|f_*-f_*^{(N)}\|_{L^\infty(D(0;1))}$; (lower left) preimage; (lower right) image by numerical backward conformal mapping, where $z_0=0.5$, and $r_{\mathrm{b}}=0.1$.}
	\label{fig:DiskToDiskBackward}
\end{figure}

\subsubsection{Cassini's oval--Disk}

The next example deals with the case where the problem region $\Omega$ is a Cassini's oval, which is defined as
\begin{align}
	\Omega
	=
	\mathcal{O}_a
	:=
	\{z\in\mathbb{C}\mid|z+1||z-1|<a^2\},
	\quad
	a>1.
\end{align}
The Cassini's oval changes its shape depending on the parameter $a$, and it becomes concave when $a\in(1,\sqrt{2}]$.
The exact forms of the forward conformal mapping $f$ and backward one $f_*$ are given by
\begin{equation}
	f(z)=\frac{az}{\sqrt{a^4-1+z^2}},
	\quad
	f_*(w)=\frac{\sqrt{a^4-1}w}{\sqrt{a^2-w^2}},
\end{equation}
where the branches of the square roots are the principal ones.
In this numerical experiment, we consider the case where $a=1.1$.
Numerical results for forward conformal mapping and the ones for backward numerical conformal mapping are shown in Figures \ref{fig:CassiniToDiskForward} and \ref{fig:CassiniToDiskBackward}, respectively.
The parameters are given as $z_0=0$, $r_{\mathrm{f}}=0.06$, and $r_{\mathrm{b}}=0.04$.
\begin{figure}[tb]
	\begin{minipage}{.535\hsize}
		\includegraphics[width=\hsize]{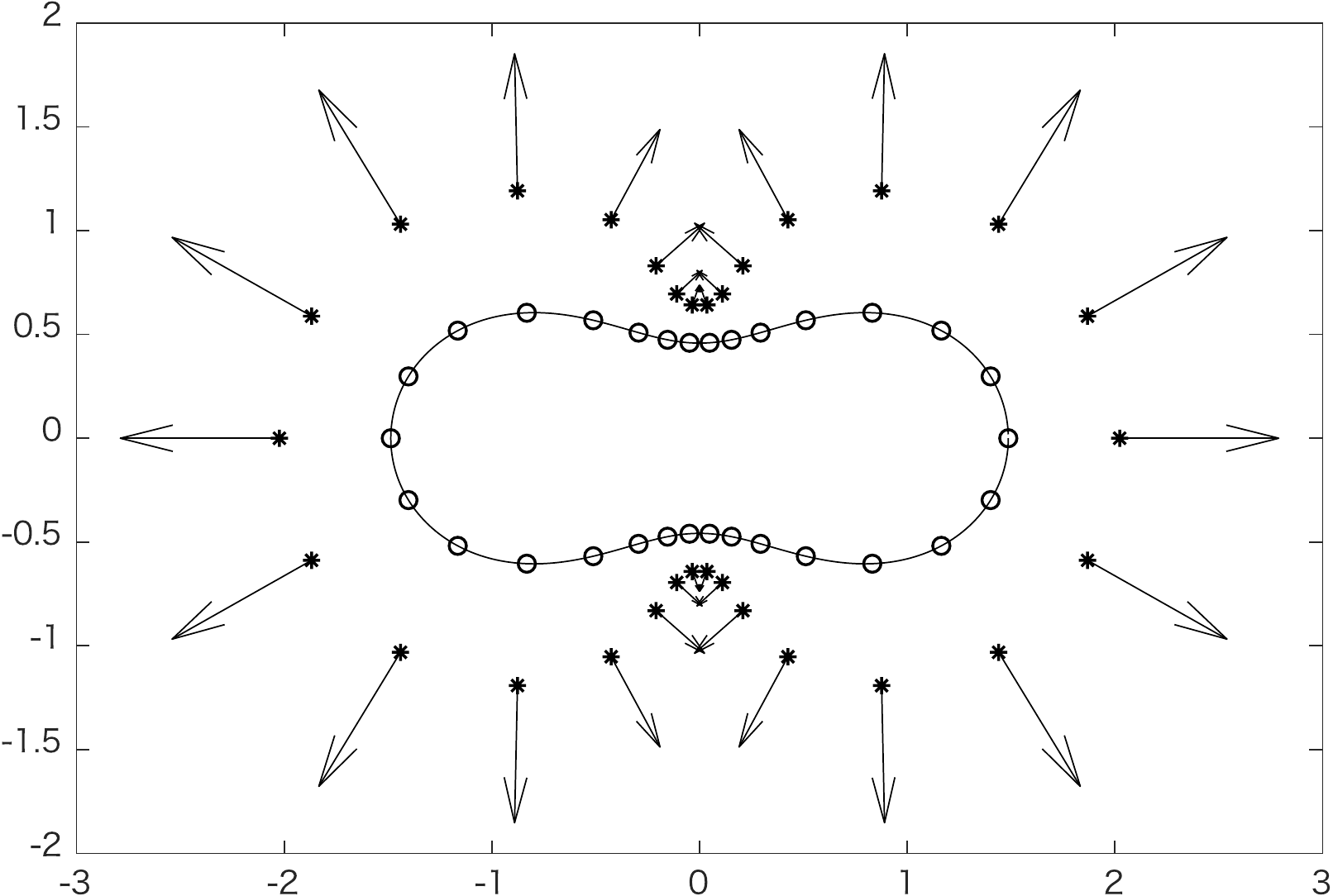}
	\end{minipage}%
	\begin{minipage}{.465\hsize}
		\includegraphics[width=\hsize]{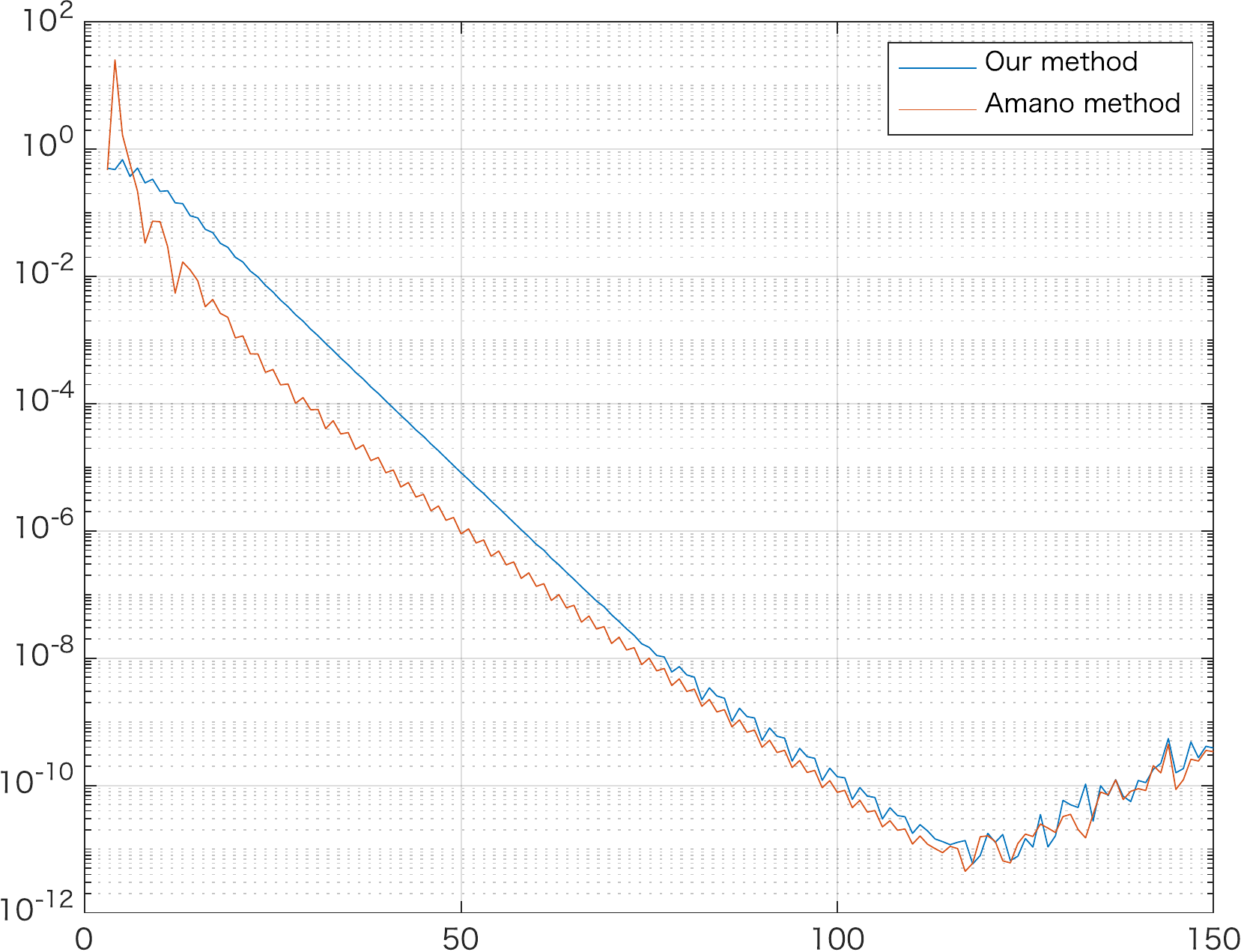}
	\end{minipage}%
	
	\begin{minipage}{.6\hsize}
		\includegraphics[width=\hsize]{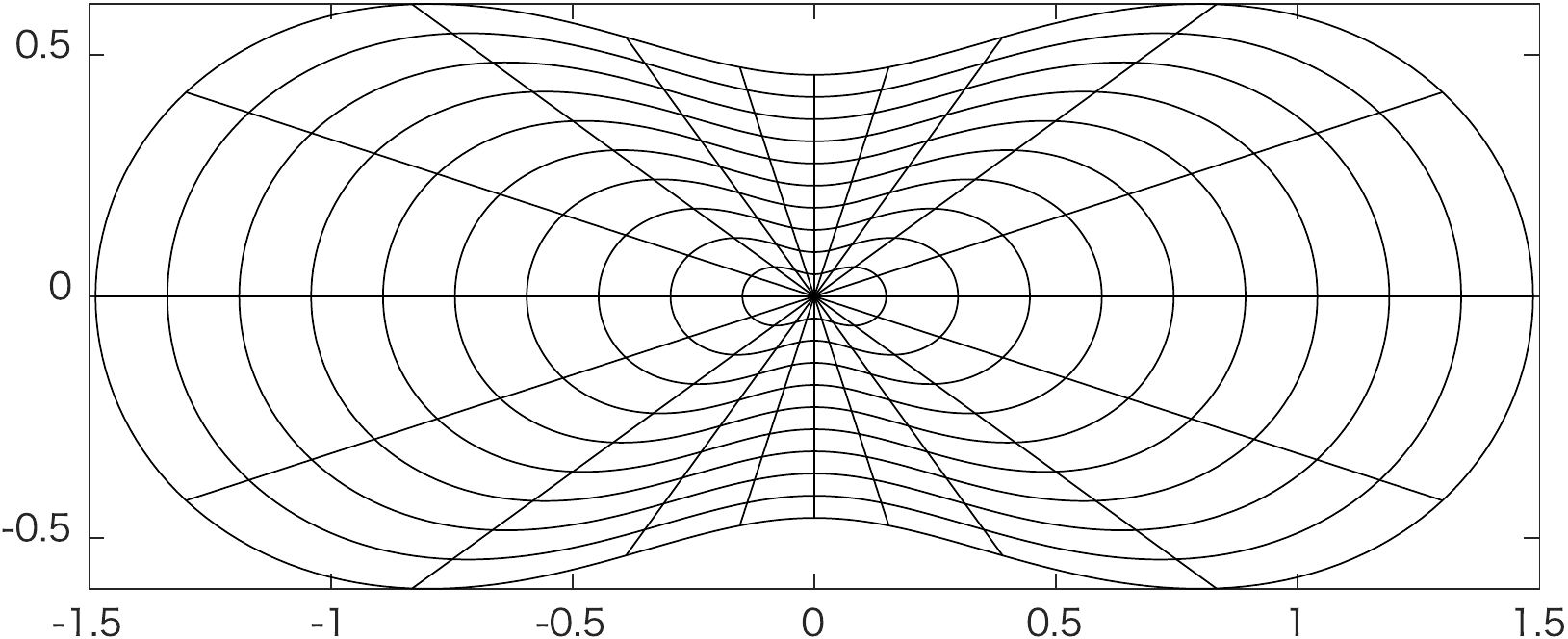}
	\end{minipage}%
	\begin{minipage}{.4\hsize}
		\includegraphics[width=\hsize]{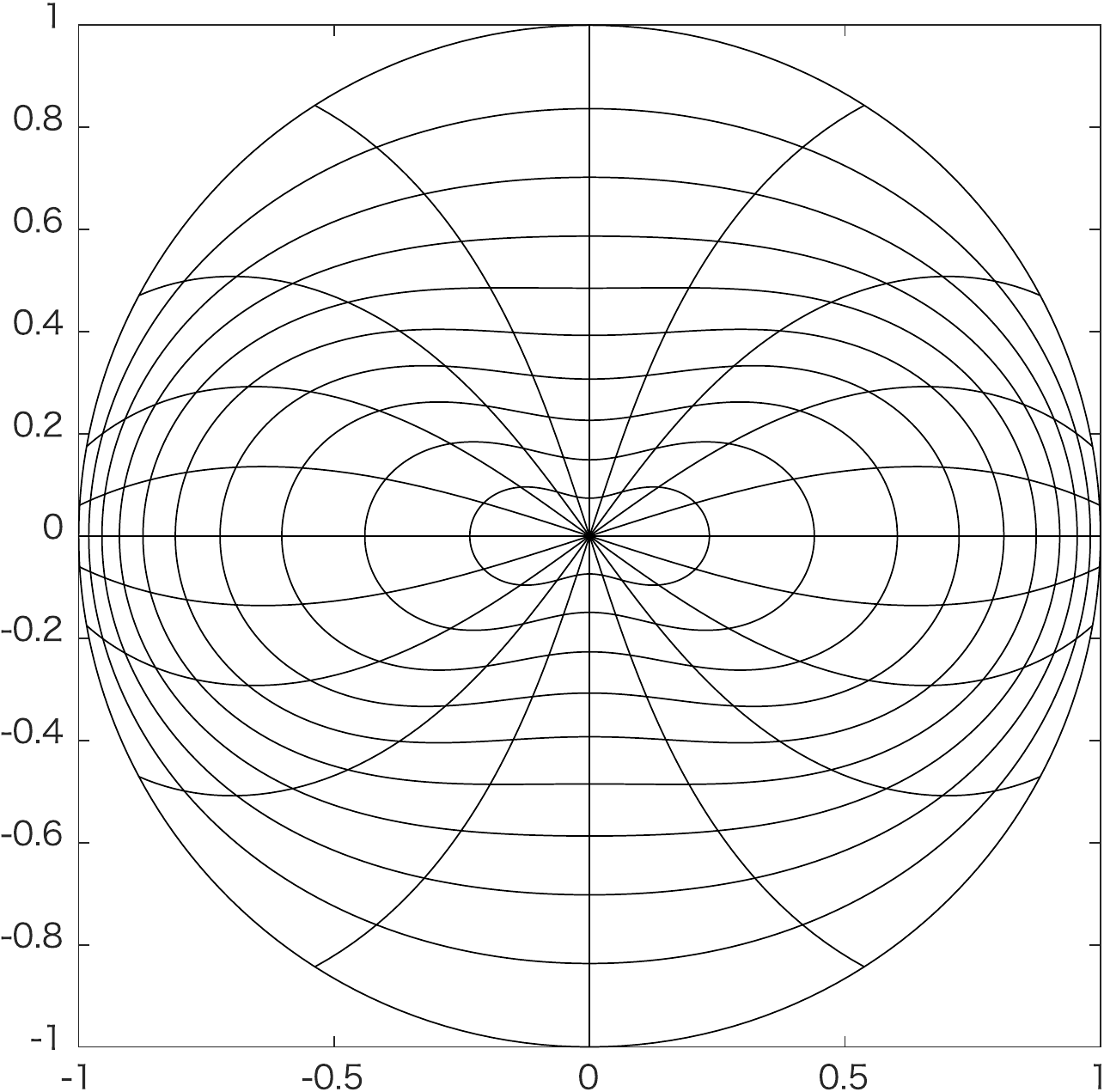}
	\end{minipage}%
	\caption{Numerical results for forward conformal mapping from the Cassini's oval onto the unit disk; (upper left) the arrangements of the singular points $\{\zeta_k\}_{k=1}^{N}$, the collocation points $\{z_j\}_{j=1}^{N}$, and the dipole moments $\{n_k\}_{k=1}^{N}$; (upper right) $N$-$\log_{10}\|f-f^{(N)}\|_{L^\infty(\Omega)}$ graph; (lower left) preimage; (lower right) image by numerical forward conformal mapping, where $a=1.1$, $z_0=0.5$, and $r_{\mathrm{f}}=0.06$.}
	\label{fig:CassiniToDiskForward}
\end{figure}
\begin{figure}[tb]
	\begin{minipage}{.445\hsize}
		\includegraphics[width=\hsize]{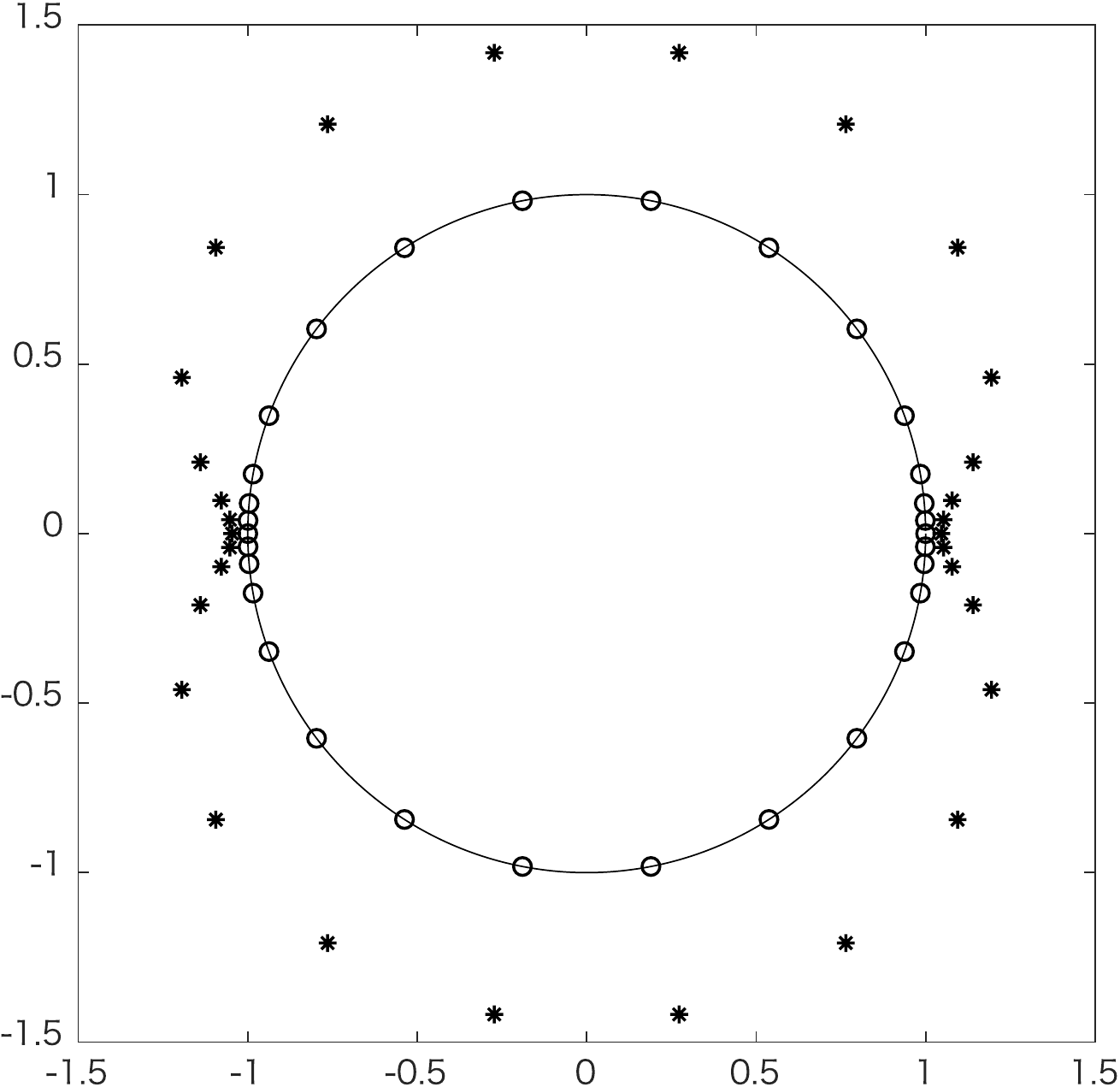}
	\end{minipage}%
	\begin{minipage}{.555\hsize}
		\includegraphics[width=\hsize]{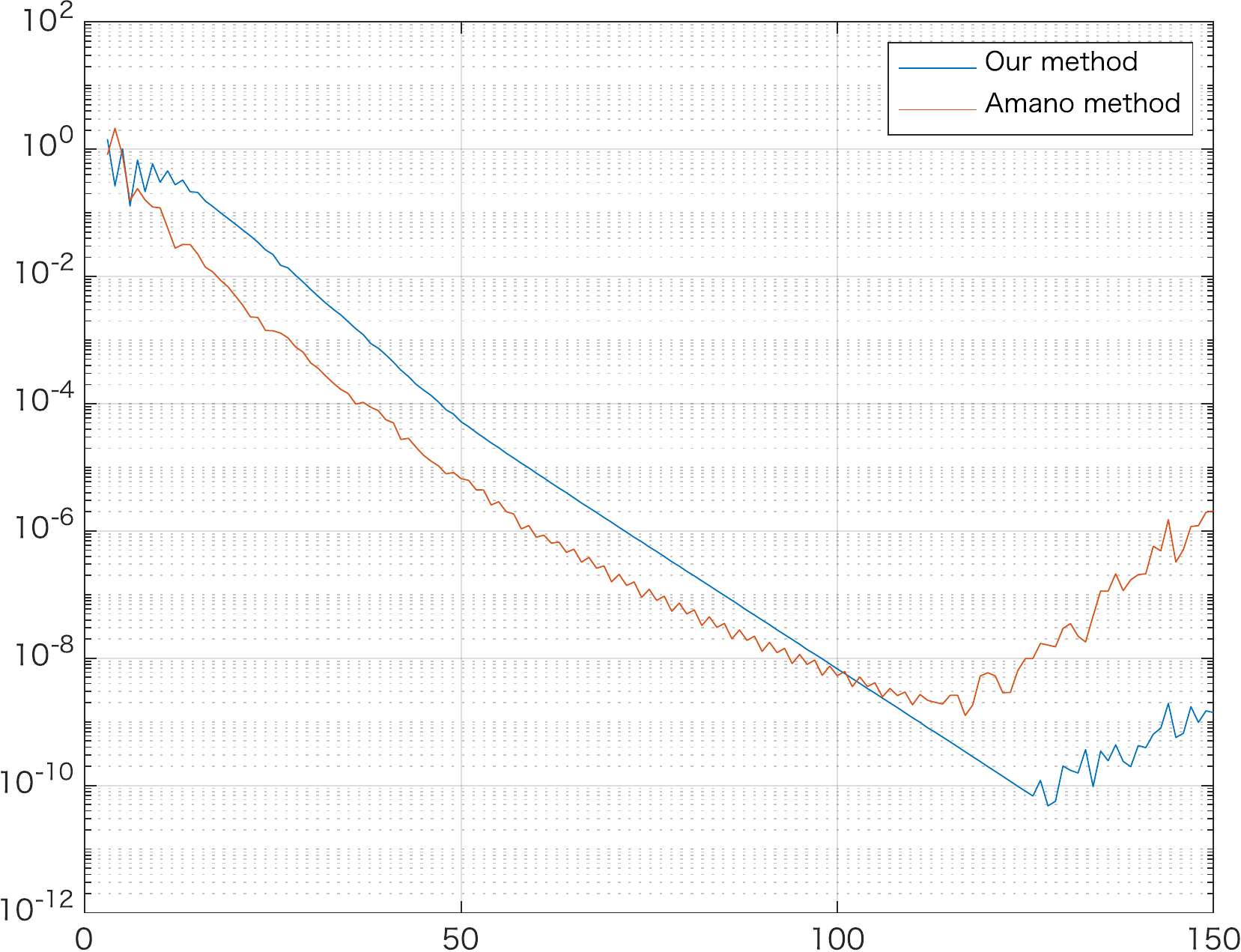}
	\end{minipage}%
	
	\begin{minipage}{.4\hsize}
		\includegraphics[width=\hsize]{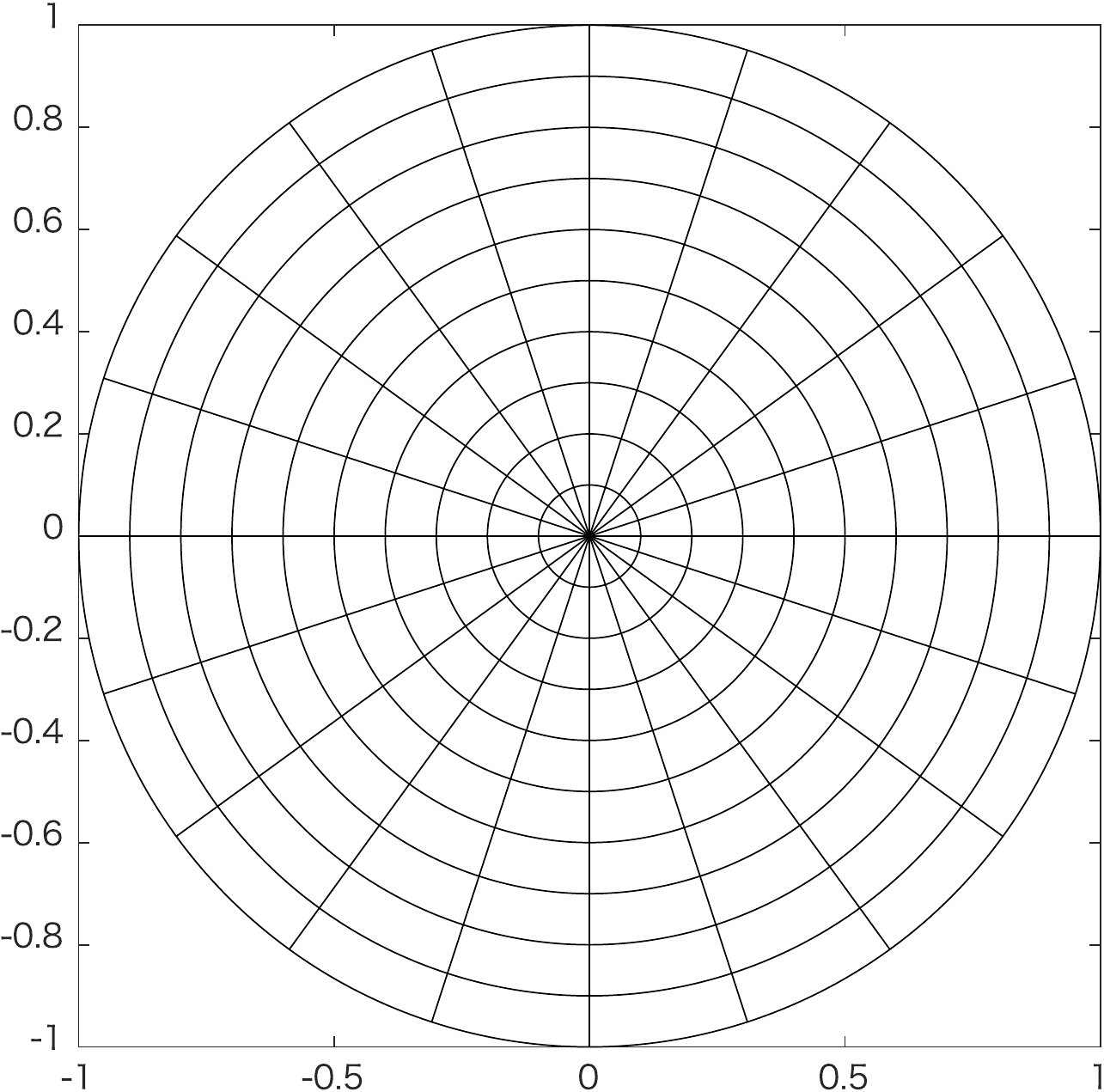}
	\end{minipage}%
	\begin{minipage}{.6\hsize}
		\includegraphics[width=\hsize]{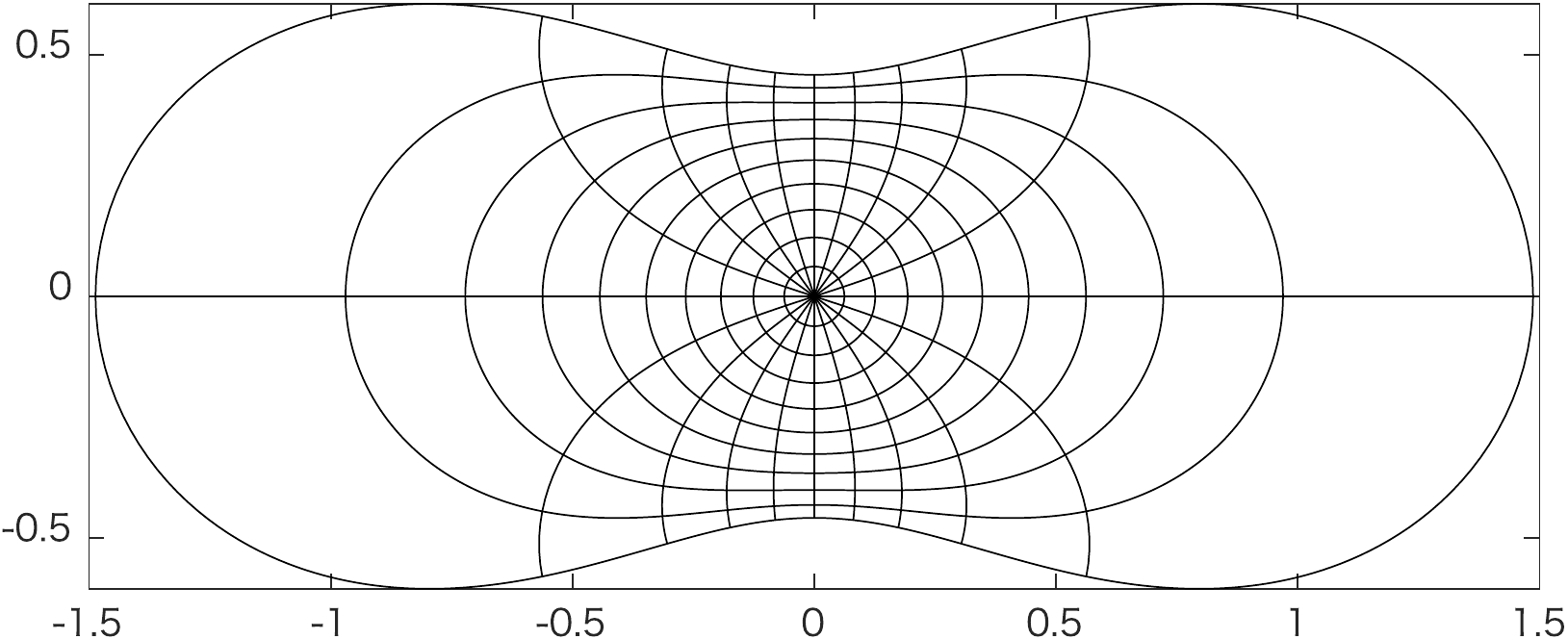}
	\end{minipage}%
	\caption{Numerical results for backward conformal mapping from the unit disk onto the Cassini's oval; (upper left) the arrangements of the singular points $\{\xi_k\}_{k=1}^{N}$, and the collocation points $\{w_j\}_{j=1}^{N}$; (upper right) $N$-$\log_{10}\|f_*-f_*^{(N)}\|_{L^\infty(D(0;1))}$; (lower left) preimage; (lower right) image by numerical backward conformal mapping, where $a=1.1$, $z_0=0$, and $r_{\mathrm{b}}=0.04$.}
	\label{fig:CassiniToDiskBackward}
\end{figure}

\subsection{Doubly-connected region}

We next consider the case where $\Omega$ is a multiply-connected region.
As an example, we deal with the following problem region $\Omega$, which we call the Cassini's frame:
\begin{equation}
	\Omega
	=
	\mathcal{F}_{a_1,b_1,a_2,b_2}
	:=
	\{z\in\mathbb{C}\mid|z^2-b_1^2|<a_1^2,\ |z^2-b_2^2|>a_2^2\}.
\end{equation}
When the condition $(a_1^4-b_1^4)/b_1^2=(a_2^4-b_2^4)/b_2^2$ holds, we can write down explicit formulae for the forward conformal mapping $f$, backward one $f_*$, and $\rho$ as follows:
\begin{equation}
	f(z)=\frac{a_1z}{\sqrt{b_1^2z^2+a_1^4-b_1^4}},
	\quad
	f_*(w)=\frac{\sqrt{a_1^4-b_1^4}w}{\sqrt{a_1^2-b_1^2w^2}},
	\quad
	\rho=\frac{a_1b_2}{a_2b_1}.
\end{equation}
Here, the square roots are interpreted to take principal values.
Numerical results are depicted in Figures \ref{fig:CassiniFrameToAnnulusForward} and \ref{fig:CassiniFrameToAnnulusBackward}.
\begin{figure}[tb]
	\begin{minipage}{.505\hsize}
		\includegraphics[width=\hsize]{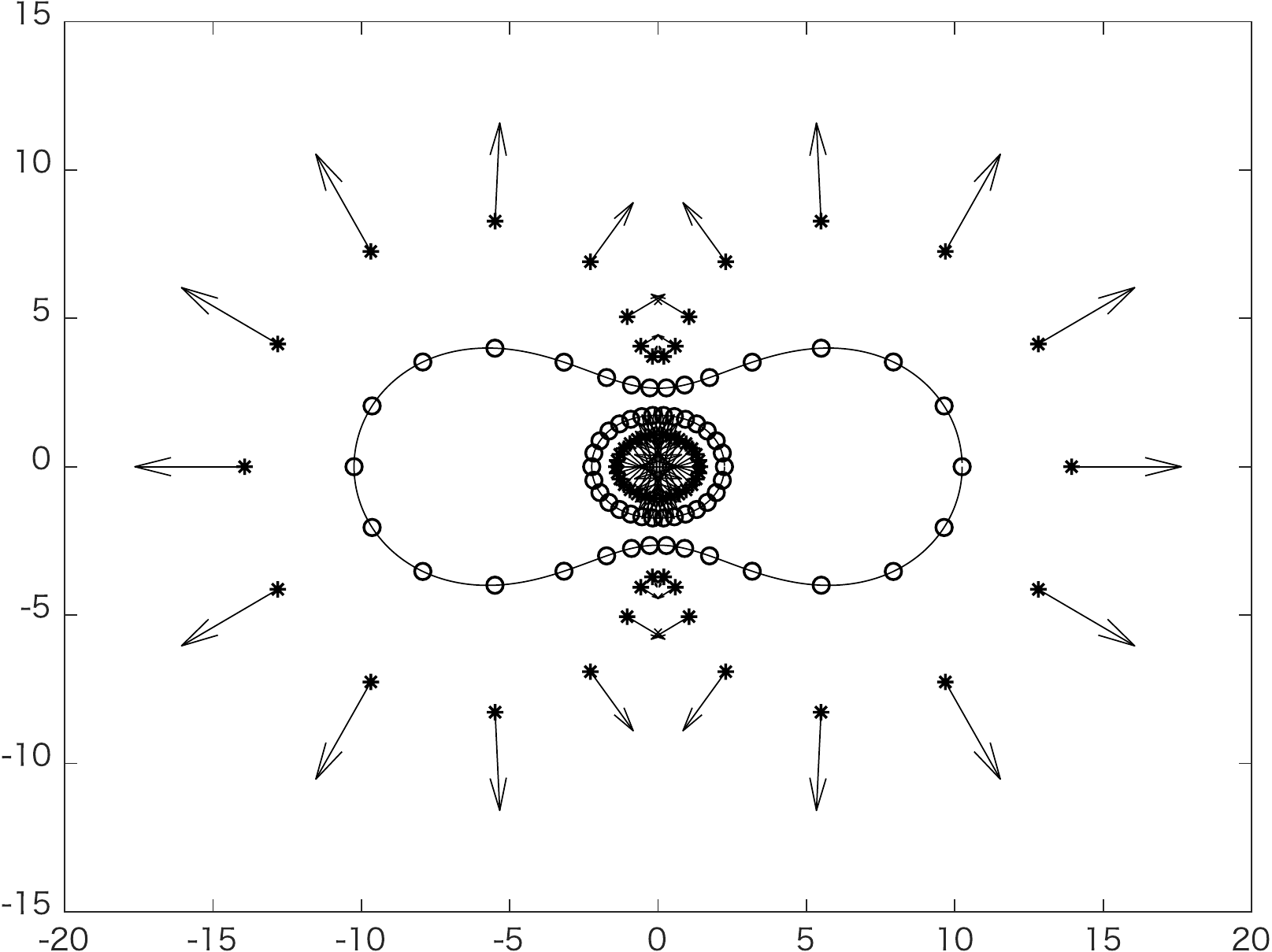}
	\end{minipage}%
	\begin{minipage}{.495\hsize}
		\includegraphics[width=\hsize]{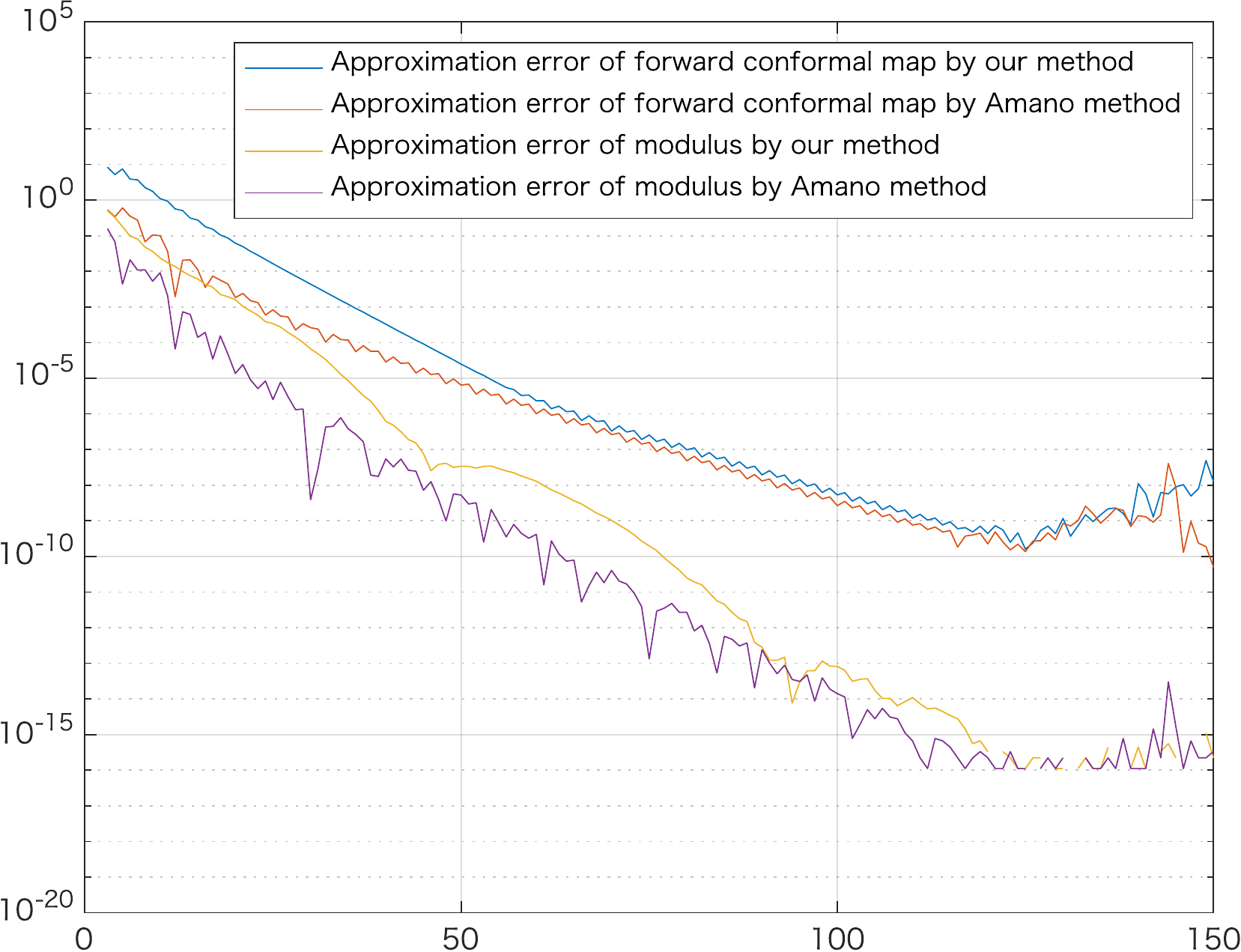}
	\end{minipage}%
	
	\begin{minipage}{.6\hsize}
		\includegraphics[width=\hsize]{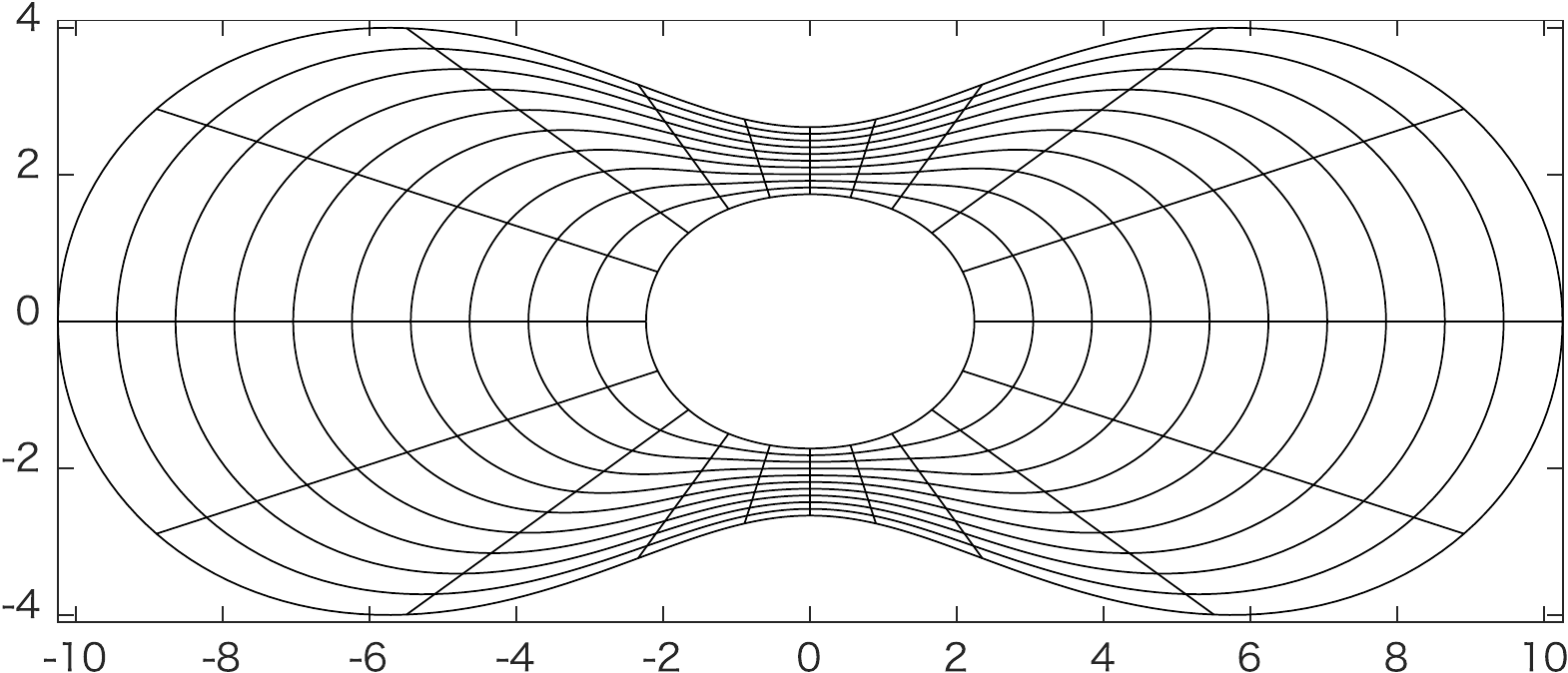}
	\end{minipage}%
	\begin{minipage}{.4\hsize}
		\includegraphics[width=\hsize]{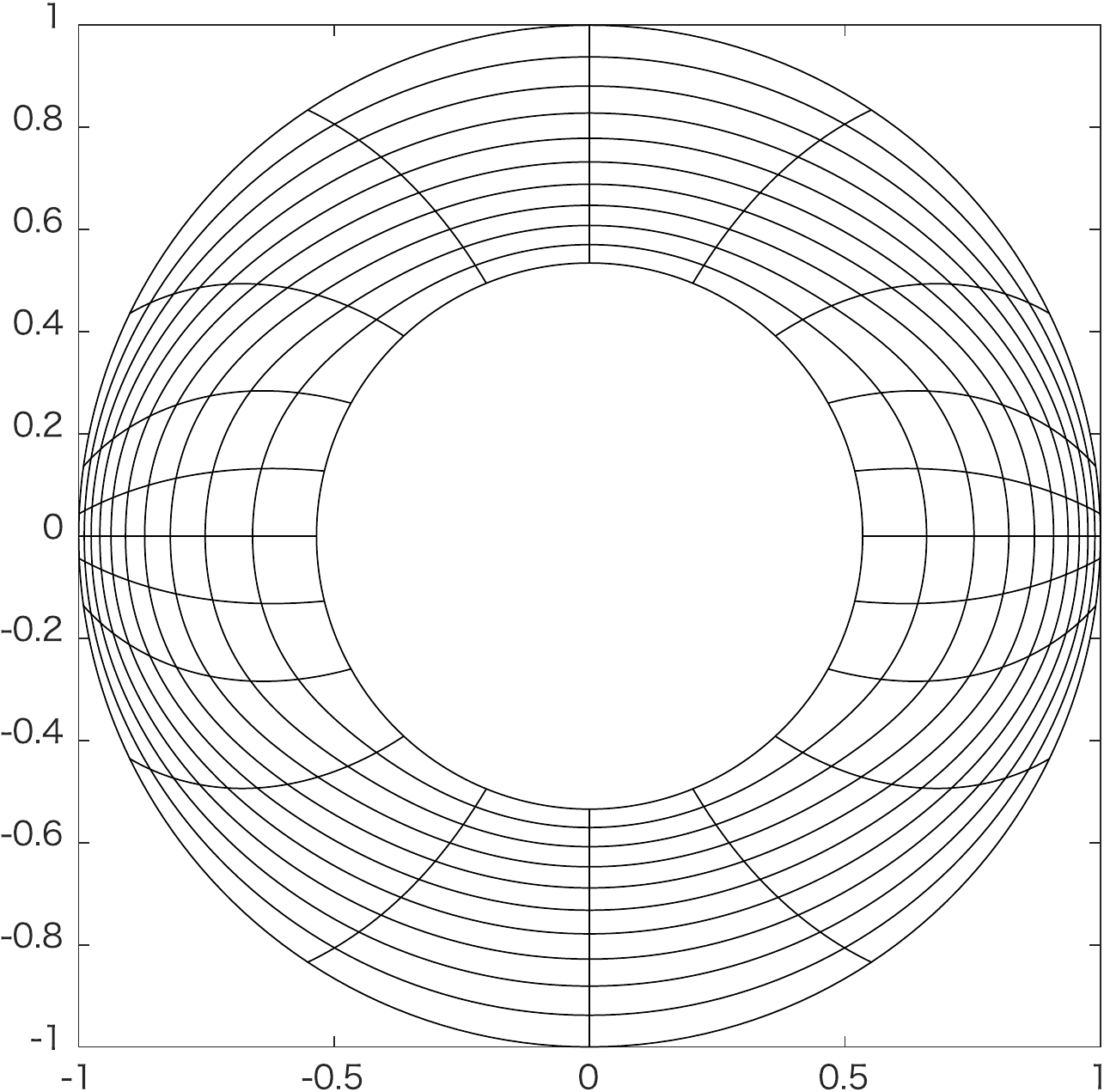}
	\end{minipage}%
	\caption{Numerical results for forward numerical conformal mapping from the Cassini's frame onto the annular region; (upper left) the arrangements of the singular points $\{\zeta_{\nu k}\}_{\nu=1,2}^{k=1,\ldots,N}$, the collocation points $\{z_{\mu j}\}_{\mu=1,2}^{j=1,\ldots,N}$, and the dipole moments $\{n_{\nu k}\}_{\nu=1,2}^{k=1,\ldots,N}$; (upper right) $N$-$\log_{10}\|f-f^{(N)}\|_{L^\infty(\Omega)}$ graph and $N$-$\log_{10}\|\rho-R^{(N)}\|_{L^\infty(\Omega)}$; (lower left) preimage; (lower right) image by numerical forward conformal mapping, where $a_1=2\sqrt{14}$, $b_1=7$, $a_2=2$, $b_2=1$, $z_0=0$, and $r_{\mathrm{f}}=0.06$.}
	\label{fig:CassiniFrameToAnnulusForward}
\end{figure}
\begin{figure}[tb]
	\begin{minipage}{.5\hsize}
		\includegraphics[width=\hsize]{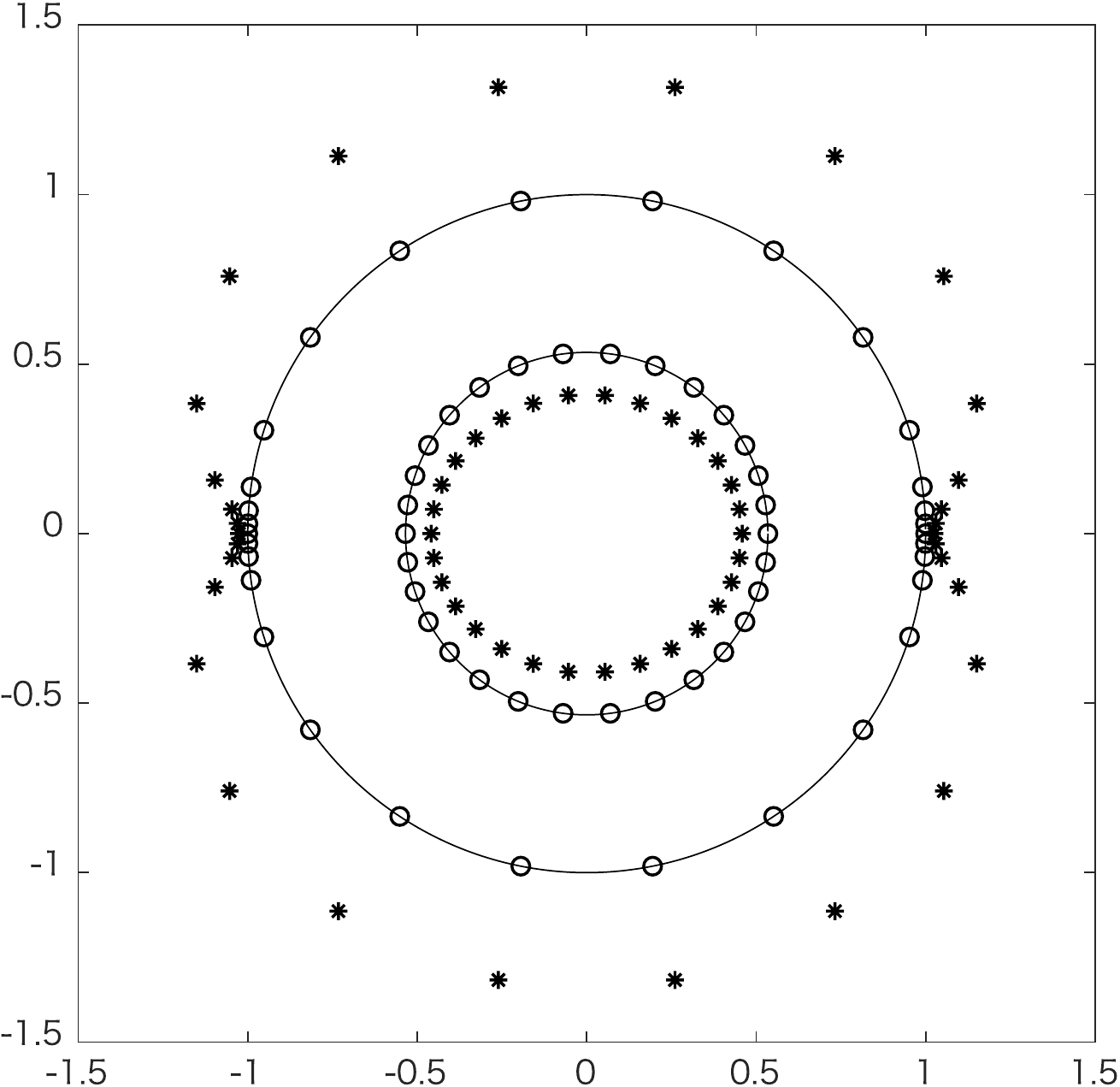}
	\end{minipage}%
	\begin{minipage}{.5\hsize}
		\includegraphics[width=\hsize]{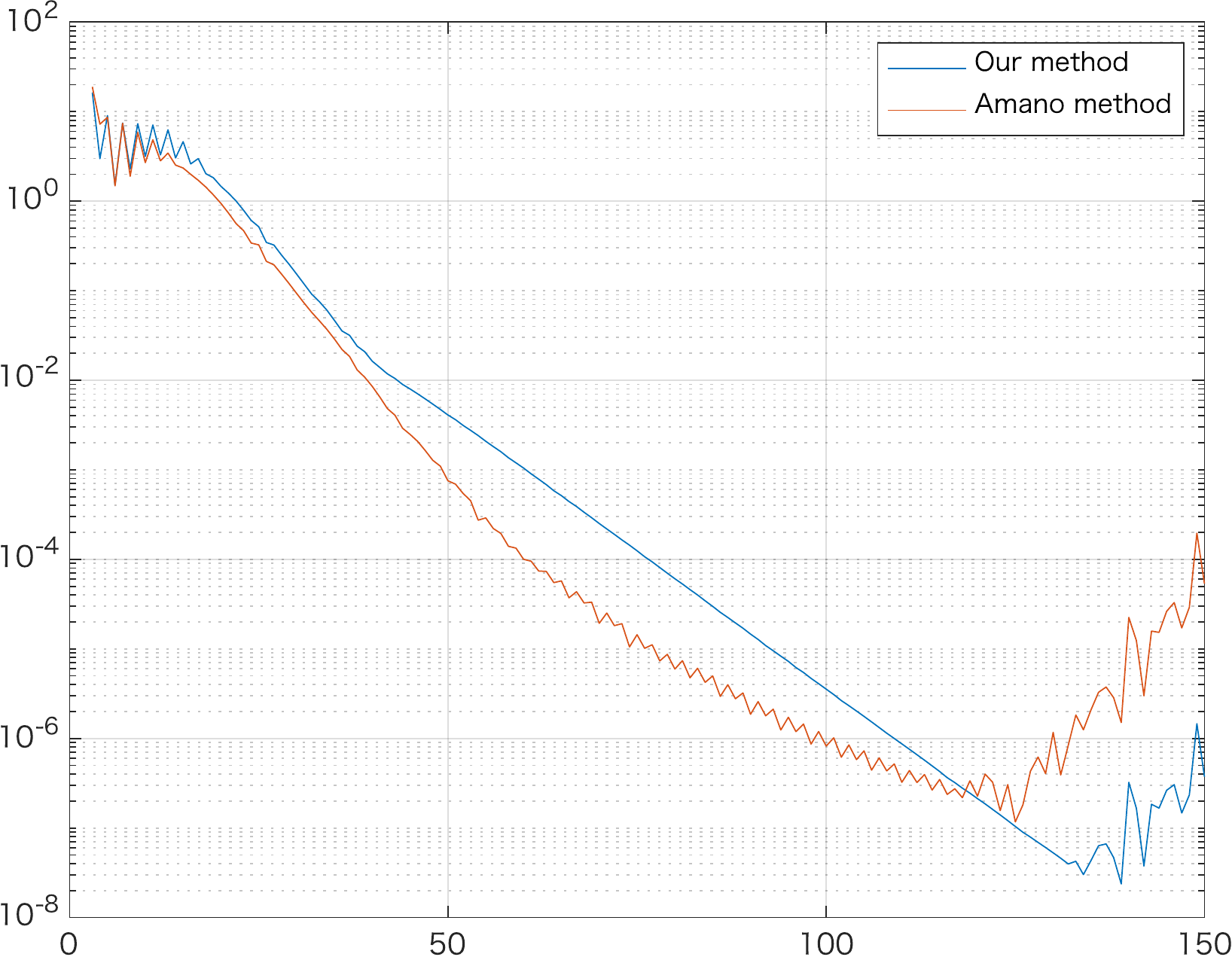}
	\end{minipage}%
	
	\begin{minipage}{.4\hsize}
		\includegraphics[width=\hsize]{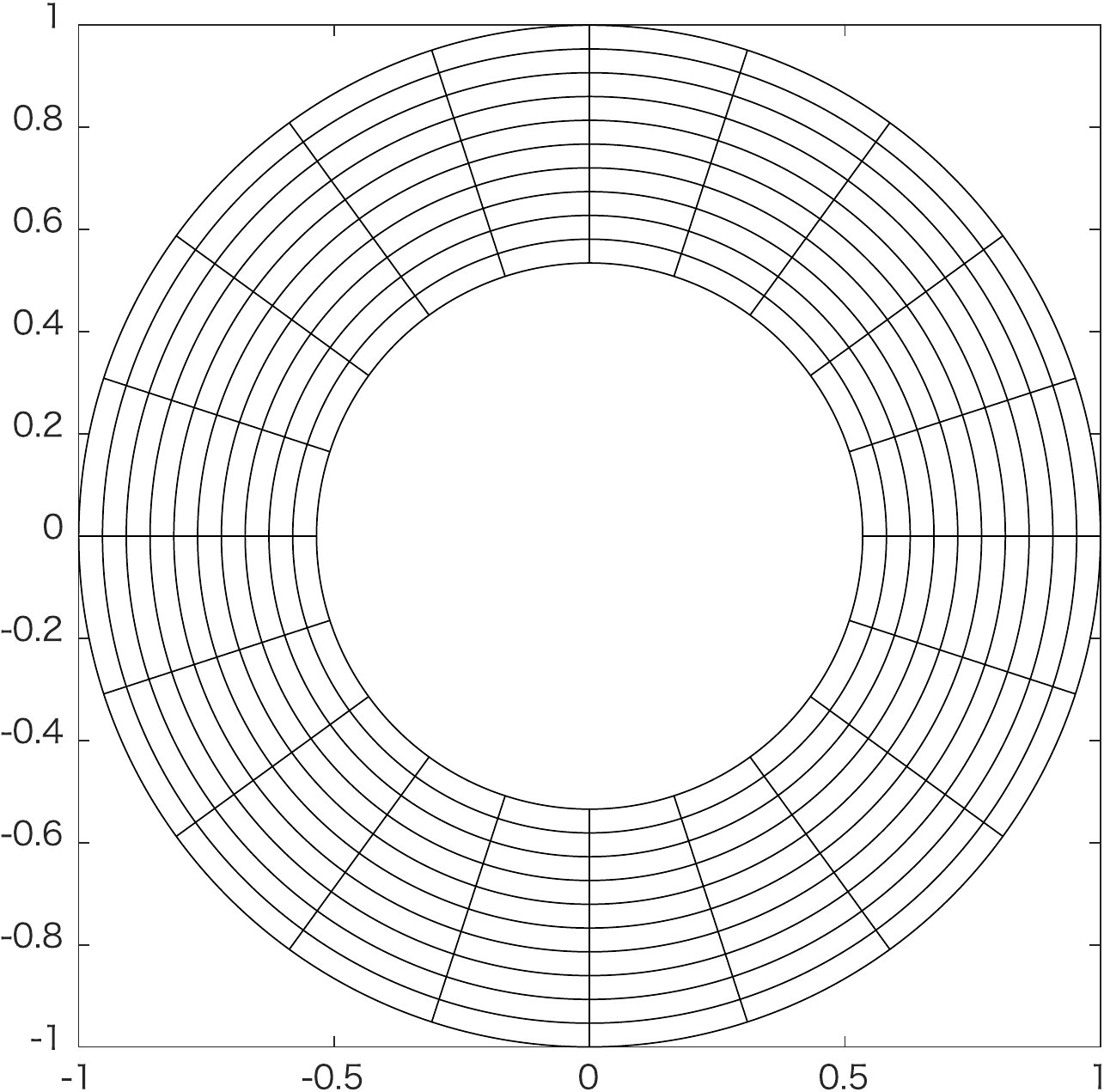}
	\end{minipage}%
	\begin{minipage}{.6\hsize}
		\includegraphics[width=\hsize]{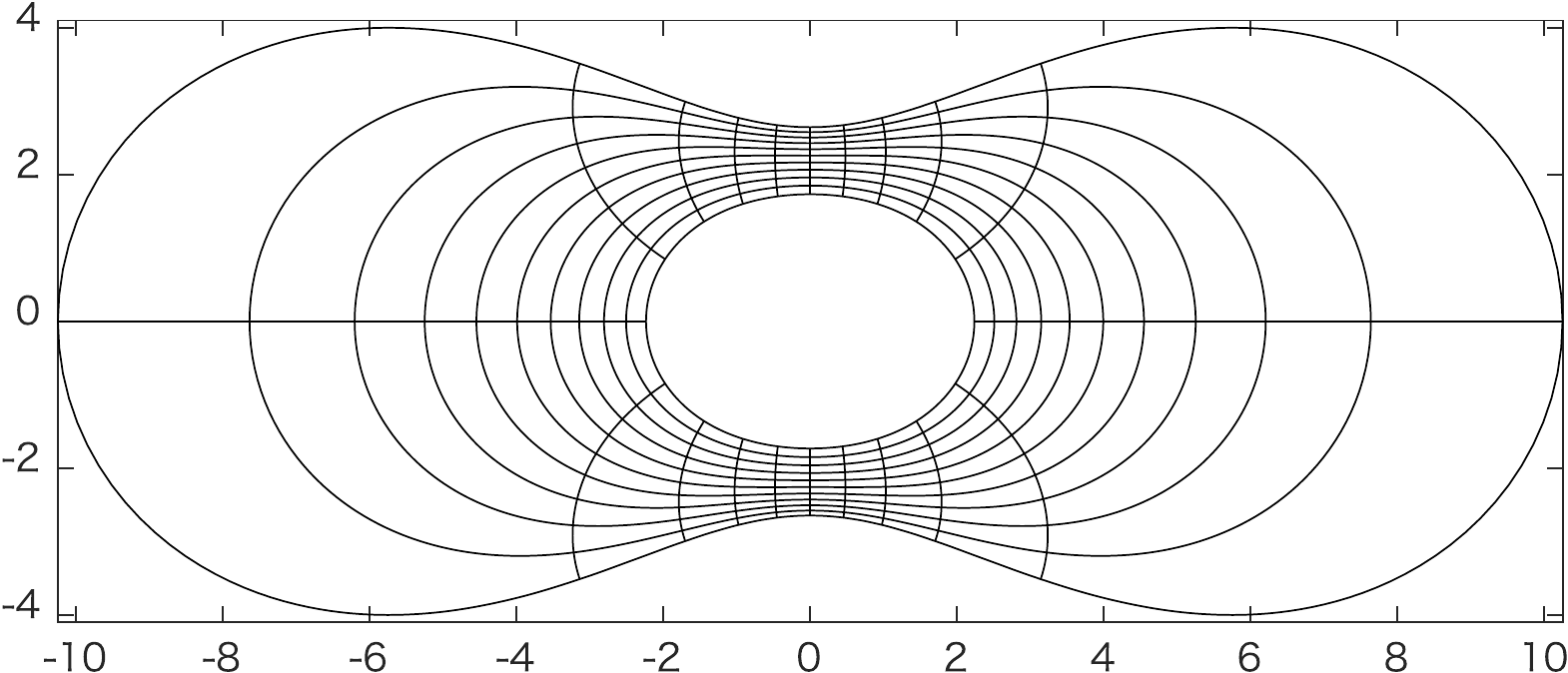}
	\end{minipage}%
	\caption{Numerical results for backward conformal mapping from the annular region onto the Cassini's frame; (upper left) the arrangements of the singular points $\{\xi_{\nu k}\}_{\nu=1,2}^{k=1,\ldots,N}$, and the collocation points $\{w_{\mu j}\}_{\mu=1,2}^{j=1,\ldots,N}$; (upper right) $N$-$\log_{10}\|f_*-f_*^{(N)}\|_{L^\infty(A(\rho,1))}$ graph; (lower left) preimage; (lower right) image by numerical backward conformal mapping, where $a_1=2\sqrt{14}$, $b_1=7$, $a_2=2$, $b_2=1$, $z_0=0$, and $r_{\mathrm{b}}=0.03$.}
	\label{fig:CassiniFrameToAnnulusBackward}
\end{figure}

In previous studies, several cases have been studied, and our method can be applied to all those situations (refer for instance to \cite{MR1310252,MR1614319,MR2931407} concerning other situations).

\section{Concluding remarks}
\label{sec:ConcludingRemarks}

In this paper, we developed a numerical scheme for bidirectional conformal mapping.
More precisely, the forward conformal mapping is computed by the dipole simulation method, and the backward one is by the complex dipole simulation method.
We established a theoretical error estimate for the approximation error of conjugate harmonic function and also gave simple mathematical observation on the arrangements of the singular points.
Moreover, several numerical results exemplify the effectiveness of our proposed method.

\subsection*{Acknowledgements}

This work is supported by JSPS KAKENHI No.~18K13455.

\bibliographystyle{plain}
\bibliography{reference}

\begin{thebibliography}{10}

\bibitem{MR1170484}
Kaname Amano.
\newblock A bidirectional method for numerical conformal mapping based on the
  charge simulation method.
\newblock {\em J. Inform. Process.}, 14(4):473--482, 1991.

\bibitem{MR1310252}
Kaname Amano.
\newblock A charge simulation method for the numerical conformal mapping of
  interior, exterior and doubly-connected domains.
\newblock {\em J. Comput. Appl. Math.}, 53(3):353--370, 1994.

\bibitem{MR1614319}
Kaname Amano.
\newblock A charge simulation method for numerical conformal mapping onto
  circular and radial slit domains.
\newblock {\em SIAM J. Sci. Comput.}, 19(4):1169--1187, 1998.

\bibitem{MR2931407}
Kaname Amano, Dai Okano, Hidenori Ogata, and Masaaki Sugihara.
\newblock Numerical conformal mappings onto the linear slit domain.
\newblock {\em Jpn. J. Ind. Appl. Math.}, 29(2):165--186, 2012.

\bibitem{MR717692}
Douglas~N. Arnold.
\newblock A spline-trigonometric {G}alerkin method and an exponentially
  convergent boundary integral method.
\newblock {\em Math. Comp.}, 41(164):383--397, 1983.

\bibitem{MR829032}
Jean-Paul Berrut.
\newblock A {F}redholm integral equation of the second kind for conformal
  mapping.
\newblock {\em J. Comput. Appl. Math.}, 14(1-2):99--110, 1986.
\newblock Special issue on numerical conformal mapping.

\bibitem{MR1357411}
Gerald~B. Folland.
\newblock {\em Introduction to partial differential equations}.
\newblock Princeton University Press, Princeton, NJ, second edition, 1995.

\bibitem{MR0396926}
Dieter Gaier.
\newblock Integralgleichungen erster {A}rt und konforme {A}bbildung.
\newblock {\em Math. Z.}, 147(2):113--129, 1976.

\bibitem{MR661073}
Dieter Gaier.
\newblock Das logarithmische {P}otential und die konforme {A}bbildung mehrfach
  zusammenh\"{a}ngender {G}ebiete.
\newblock In {\em E. {B}. {C}hristoffel ({A}achen/{M}onschau, 1979)}, pages
  290--303. Birkh\"{a}user, Basel-Boston, Mass., 1981.

\bibitem{MR0301176}
John~K. Hayes, David~K. Kahaner, and Richard~G. Kellner.
\newblock An improved method for numerical conformal mapping.
\newblock {\em Math. Comp.}, 26:327--334; suppl., ibid. 26 (1972), no. 118,
  loose microfiche suppl. A1--B14, 1972.

\bibitem{MR545882}
Peter Henrici.
\newblock Fast {F}ourier methods in computational complex analysis.
\newblock {\em SIAM Rev.}, 21(4):481--527, 1979.

\bibitem{MR822470}
Peter Henrici.
\newblock {\em Applied and computational complex analysis. {V}ol. 3}.
\newblock Pure and Applied Mathematics (New York). John Wiley \& Sons, Inc.,
  New York, 1986.
\newblock Discrete Fourier analysis---Cauchy integrals---construction of
  conformal maps---univalent functions, A Wiley-Interscience Publication.

\bibitem{MR615896}
D.~M. Hough and N.~Papamichael.
\newblock The use of splines and singular functions in an integral equation
  method for conformal mapping.
\newblock {\em Numer. Math.}, 37(1):133--147, 1981.

\bibitem{MR712114}
D.~M. Hough and N.~Papamichael.
\newblock An integral equation method for the numerical conformal mapping of
  interior, exterior and doubly-connected domains.
\newblock {\em Numer. Math.}, 41(3):287--307, 1983.

\bibitem{MR1362387}
M.~Katsurada and H.~Okamoto.
\newblock The collocation points of the fundamental solution method for the
  potential problem.
\newblock {\em Comput. Math. Appl.}, 31(1):123--137, 1996.

\bibitem{MR991024}
Masashi Katsurada.
\newblock A mathematical study of the charge simulation method. {II}.
\newblock {\em J. Fac. Sci. Univ. Tokyo Sect. IA Math.}, 36(1):135--162, 1989.

\bibitem{katsurada1998mathematical}
Masashi Katsurada.
\newblock A mathematical study of the charge simulation method by use of
  peripheral conformal mappings.
\newblock {\em Mem. Inst. Sci. Tech. Meiji Univ.}, 37(8):195--211, 1998.

\bibitem{MR2491542}
Mohamed M.~S. Nasser.
\newblock Numerical conformal mapping via a boundary integral equation with the
  generalized {N}eumann kernel.
\newblock {\em SIAM J. Sci. Comput.}, 31(3):1695--1715, 2009.

\bibitem{MR2805493}
Mohamed M.~S. Nasser.
\newblock Numerical conformal mapping of multiply connected regions onto the
  second, third and fourth categories of {K}oebe's canonical slit domains.
\newblock {\em J. Math. Anal. Appl.}, 382(1):47--56, 2011.

\bibitem{MR3071407}
Mohamed M.~S. Nasser and Fayzah A.~A. Al-Shihri.
\newblock A fast boundary integral equation method for conformal mapping of
  multiply connected regions.
\newblock {\em SIAM J. Sci. Comput.}, 35(3):A1736--A1760, 2013.

\bibitem{ogata2015dipole}
Hidenori Ogata.
\newblock Dipole simulation method for two-dimensional potential problems.
\newblock {\em NOLTA, IEICE}, 5(1):2--14, 2014.

\bibitem{MR829034}
Lothar Reichel.
\newblock A fast method for solving certain integral equations of the first
  kind with application to conformal mapping.
\newblock {\em J. Comput. Appl. Math.}, 14(1-2):125--142, 1986.
\newblock Special issue on numerical conformal mapping.

\bibitem{MR3576615}
Koya Sakakibara.
\newblock Analysis of the dipole simulation method for two-dimensional
  {D}irichlet problems in {J}ordan regions with analytic boundaries.
\newblock {\em BIT}, 56(4):1369--1400, 2016.

\bibitem{MR3448858}
Koya Sakakibara and Masashi Katsurada.
\newblock A mathematical analysis of the complex dipole simulation method.
\newblock {\em Tokyo J. Math.}, 38(2):309--326, 2015.

\bibitem{MR0207240}
George~T. Symm.
\newblock An integral equation method in conformal mapping.
\newblock {\em Numer. Math.}, 9:250--258, 1966.

\end{thebibliography}
\end{document}